\documentclass[12pt]{amsart}
\usepackage{amsmath,amsxtra} 
\usepackage{amssymb,amsthm}
\usepackage{latexsym}
\usepackage{enumitem}

\usepackage{graphicx}
\usepackage{psfrag,color,float}

 \usepackage{mathrsfs}

\newtheorem{theorem}[subsection]{Theorem}
\newtheorem{prop}[subsection]{Proposition}
\newtheorem{lemma}[subsection]{Lemma}
\newtheorem{corollary}[subsection]{Corollary}

\newcommand{\n}{\newline}

\theoremstyle{remark}
\newtheorem{remark}{Remark} 

\theoremstyle{example}
\newtheorem{example}{Example} 

\theoremstyle{question}

\DeclareMathOperator*{\slim}{s\mbox{-}l.i.m.}
\DeclareMathOperator*{\mylim}{\mbox{}l.i.m.} 
\DeclareMathOperator*{\sslim}{s-lim} 

\newcommand{\norm}[1]{ \left\|  #1 \right\| }

\definecolor{orange}{rgb}{0.995, 0.75, 0.35}
\definecolor{purple}{rgb}{0.7, 0.2, 0.5}
\definecolor{royalblue}{rgb}{0.2, 0.7, 0.8}
\definecolor{darkgreen}{rgb}{0.2,0.725,0.25}

\newcommand{\cal}{\mathcal}

\newcommand{\R}{\mathbb{R}}
\newcommand{\C}{\mathbb{C}}

\def\fsh{f^\sharp}
\def\cF{{\cal F}}

\def\cV{\mathcal{V}}

 \global\long\def\Fo{\operatorname{\mathscr{F}}}

\def\ch{\cosh}
\def\th{\tanh}
\def\sech{\mathrm{sech}}

\def\al{\alpha}
\def\de{\delta}
\def\eps{\epsilon}
\def\lam{\lambda}
\def\veps{\varepsilon}
\def\vphi{\varphi}

\def\De{\Delta}

\def\Om{\Omega}
\newcommand{\ran}{\mathrm{Ran}\,}

\def\cB{\mathcal{B}}
\def\cD{\mathcal{D}}

\def\cE{\mathcal{E}}

\def\cM{\mathcal{M}}

\def\sH{\mathscr{H}}

\def\iy{\infty}
\def\inv{^{-1}}

\def\supp{\mathrm{supp}}

\newcommand{\gv}[1]{\pmb{#1}}

\newcommand{\gvsig}{\gv{\sigma}}

\newcommand{\la}{\langle}
\newcommand{\ra}{\rangle}

\def\sideremark#1{\ifvmode\leavevmode\fi\vadjust{\vbox to0pt{\vss
 \hbox to 0pt{\hskip\hsize\hskip1em
\vbox{\hsize2cm\tiny\raggedright\pretolerance10000
 \noindent #1\hfill}\hss}\vbox to8pt{\vfil}\vss}}}%

{\end{list}}

\begin{document}

\title[1D Schr\"odinger Operators]%
      {Perturbed Fourier transform associated with  Schr\"odinger Operators} 
\author[Shijun Zheng]%
       {Shijun Zheng}
\address[S. Zheng]{Department of Mathematical Sciences\\
        Georgia Southern University, GA 30460}
\email{szheng@GeorgiaSouthern.edu}
\date{November 30, 2022}
\keywords{spectral theory, Schr\"odinger operator, scattering}  
\subjclass[2020]{Primary: 35J10, 35P10; Secondary: 35P25, 42B37, 47A40}   

\begin{abstract}
We give an exposition on  the $L^2$ theory of the perturbed Fourier transform associated with a
Schr\"odinger operator $H=-d^2/dx^2 +V$ on the real line,  
where $V$ is a real-valued  \mbox{finite} measure.  
In the case $V\in L^1\cap L^2$, we explicitly 
define the perturbed Fourier transform $\mathcal{F}$ for $H$ and obtain an 
eigenfunction expansion theorem for square integrable functions. 
This provides a complete proof of the inversion formula  for $\cF$ 
that covers the class of short range potentials in   
$(1+|x|)^{-\frac12-\eps} L^2 $. 
 Such paradigm has  applications  in the study of scattering problems
 in connection with the spectral properties  
 and asymptotic completeness of the wave operators. 
\end{abstract}

\maketitle 

\setcounter{tocdepth}{1} 
\tableofcontents

\section*{List of Symbols} \bigskip

\begin{tabular}{cp{0.806\textwidth}}
$\R_*$ & $\R_*=\R\setminus \{0\}$ \\
$L^1_k$      &      $L^1_k=\{\text{$f$ real-valued}: \int_\R(1+|x|^k) |f| dx<\iy\}$ \\ 
$\cM_k$ &  $\cM_k=\{ \mu \ \text{finite measure}:  \int_\R (1+|x|^k) |\mu| (dx)<\iy\}$\\
$H_0$    &    $H_0=\text{closure of   ${-}d^2/dx^2$} $\\
$H_V$  &  $H_V=H_0+V$, Schr\"odinger operator with a potential\\ 
$\slim$ & {strong limit in $L^2$}    \\  
$R_0(z)$ &      $R_0(z)=(H_0-z)\inv$, resolvent for the laplacian\\
$R(z)$ &    $R(z)=(H-z)\inv$,    resolvent for the operator $H$ \\ 
$J_z^\pm$ & Jost functions \\
$\vphi_k^\pm(x)$  & scattering solutions  \\
  $ e_\pm(x,\xi)$ & solutions to  Lippmann-Schwinger equation\\
  $\Fo_\pm$ &   perturbed Fourier transforms   \\
$\Fo^*_\pm $ &   adjoints to the perturbed Fourier transforms \\
$W_{\pm}$ & $W_\pm=s \mbox{-}\lim_{t\rightarrow \pm\infty} e^{itH} e^{-itH_0}$ wave operators\\
 \end{tabular}

 \bigskip

\section{Introduction} 
Consider the Schr\"odinger operator $H=H_0 +V$ on $\mathbb{R}$,
where $H_0$ is the selfadjoint extension of $-d^2/dx^2$ in $L^2(\R)$, and the potential  $V$ is a real-valued distribution. 
In Fourier analysis, it is well-known that a square integrable function admits 
an expansion with the exponentials as eigenfunctions of $H_0$. A natural question is: Can we represent an $L^2$ function in terms of
``eigenfunctions'' of $H=H_V$, a perturbation of the Laplacian?     
Eigenfunction expansion problems have a long history in connections with spectral  and
scattering theory arising in analysis of mathematical physics. 

If $V$: $\R\to \R$ is in $L^1\cap L^2$ 
or if $V$ is   measure-valued satisfying  
$\int_\R (1+x^2) d\mu_{|V|} < \infty$, 
(here and elsewhere $d\mu_V:=V(dx) $ denotes a real Borel measure $V$ on the line),
then it is known that (\cite{Ch01}, \cite{GH98}, \cite{Z04a}) 
the essential spectra of $H$ and $H_0 $ coincide and there is no
singular continuous spectrum; the wave operators $W_{\pm}=
s-\lim_{t\rightarrow \pm \infty}e^{ i t H } e^{-i t H_0}$ exist and are
complete.   
In this paper,  we  first give an expository  in \S \ref{s:preliminV} to \S \ref{s:perturb:F} on the perturbed eigenfunction expansion 
associated to $H$, where  $V$ is a real-valued finite (Borel) measure on the line satisfying  
\begin{align}\label{eV:wei2}
\int_\R (1+x^2) |V| (dx)<\iy\,,
\end{align}
based on  \cite{Bra85,GlowKu2020,GH98,Herc89,KuZn17,Tip90,Ti46,Zheng2010i}.
The second part of the paper (\S \ref{s:L2} to \S \ref{s:main:proof}) 
is denoted to treating the class of potentials 
 when $V$ is a real-valued function in  $L^1\cap L^2$, which is a more general condition than \eqref{eV:wei2} for integrable functions. 
 For instance, the class $(1+|x|)^{-\frac12-\eps} L^2\subset L^1\cap L^2$ but not included in 
 $\cM_2$ where $\cM_k:=\{ \mu \in \cM:  \int_\R (1+|x|^k) |\mu| (dx)<\iy\}$ and $\cM=\cM_0$ is the space of all finite Borel measures on $\R$. 
 
 For $V\in L^1\cap L^2$, we  prove an eigenfunction expansion theorem (Theorem \ref{th:Plancherel})  by means of  the perturbed Fourier transform for $H$. 
In particular, we give a complete proof of \cite[Theorem 1]{Z04a} based on scattering eigenfunction expansion method and the completeness property for $W_\pm$. Our approach is of
  time-dependent pattern  that is different than those taken up in \cite{Ch01,GH98}. 
 As a consequence, we obtain an explicit formula for the integral kernels of  spectral operators $\vphi(H)$  in Proposition \ref{c:vphi(H):ker}.  
Recently, the spectral  theory associated with $H$  {has} been further developed in the study of  harmonic analysis and PDE  \cite{BattyChen20,BZ2010,ChenPu2019,DanFan06,GSch04,OZh06,OZh08,Zheng2010i}. 


The proof of Theorem \ref{th:Plancherel} is mainly motivated by Ikebe's  treatment \cite{Ikebe60}  for the eigenfunction expansions in three dimensions.  
For $x,\xi\in\R$ the generalized eigenfunctions $e(x,\xi)$ of $H$ satisfy 
\begin{equation}
(-d^{\,2}/dx^2 + V(x)) e(x,\xi) = \xi^2 e(x,\xi) \label{S_eq} 
\end{equation} 
which are called {\em scattering eigenfunctions} that 
  obey the integral equation (\ref{LS_eq}) of Lippmann-Schwinger type. 
    This enables us to define  $\cal{F}f$, the perturbed Fourier transform associated with $H$. 
    A perturbative 
     approach  leads us to studying the properties for the Green's function (Lemma \ref{L:R(z)})
     and its Fourier transform (Lemma \ref{lem:F-Green}).
     Thus we are able to establish    
\begin{align}\label{Fab:plancherel}
\int_{\alpha\le k^2\le \beta} \vert \cal{F}f(k)\vert^2dk= \Vert P_{[\alpha,\beta]}(H)f \Vert_{2}^2\,
\end{align}
in Lemma \ref{lem:Pf}, which is one  main ingredient for the proof of Theorem \ref{th:Plancherel}, and where 
$P_E$ is the spectral projection operator for $H$
corresponding to  the set $E\subset \R$.   
 The generalized Fourier transform  $\cF f$ takes the form  of \eqref{cF:L2}, which  
coincides with $\Fo_+{f}$ from (\ref{Fo+:f}) when $V\in \cV_0=L^1_2\cap L^2$, 
see the remarks in the end of section \ref{s:preliminV}.  

Theorem \ref{th:Plancherel} is essentially equivalent to the Plancherel  formula \eqref{Fab:plancherel}, which shows that  
$\cF^*\cF=P_{ac}$, that is, an inversion formula on $\sH_{ac}$, the absolutely continuous subspace for $H$.  
Moreover, using wave operator method we  prove in Proposition \ref{Th:F-surj}  that $\cF$ is a bijective isometry from $\sH_{ac}$ onto $L^2$.
Thus   $\cF^*$ is its inverse from $L^2$ onto $\sH_{ac}$ and it follows that $\cF\cF^*=I$ on $L^2$, see Proposition \ref{c:wave:F*F0}.  
These are the complete $L^2$ theory for  $\cF$ and its adjoint $\cF^*$.   
Our treatment in sections \ref{s:L2} to \ref{s:Om:F:surj} is based on a delicate improvement of the analogous results for a bounded decaying potential \cite{Epp}.
The  $L^2$ theory above 
 was originally proven in three dimensions 
\cite{Ikebe60,Povzner53}
and later  extended  to  higher dimensions \cite{Thoe67} as well as abstract setting \cite{Kato66,Kuroda67,RS}. 

As a result of Theorem \ref{th:Plancherel} and Proposition \ref{c:wave:F*F0}, we  obtain 
a kernel representation   $K(x,y):=\vphi(H)(x,y)$ for  a continuous and compactly supported function $\vphi$ 
 in Propositions \ref{c:vphi(H):ker} and  \ref{p:phi(H)(xy)},  %
which is simpler  than those found in e.g.,  
 \cite{EcKMTe2014,GH98,Kuroda67,PanU16}.   
Such kernel expression has been useful in the study of  theory of function spaces associated to $H$ \cite{BZ2010,OZh06,OZh08,Zheng2010i}.

In section \ref{s:waveOp:scatter} we further consider the scattering problem for $H$. 
Denote $U_0(t):=e^{-itH_0}$  the propagator for the free hamiltonian
and  $U(t):=e^{-itH}$ for the perturbed hamiltonian.  
We prove the linear scattering  in 
Theorem \ref{th:Omega}, which states that the solution of (\ref{eU:H_V})  scatters in the space of absolutely continuous vectors:
\begin{align}\label{eU:H_V}
& i\, \partial_t \psi= H \psi,\quad\;  
\psi(0,x)=\phi(x) \in \mathcal{D}(H)\,.
\end{align}
Namely,  for any $\phi\in \sH_{ac}$\,, there exists $\phi_+$ in $L^2$ such that 
\begin{align}
 &U(t)\phi =  U_0(t) \phi_+  +o(1)\quad in \  L^2  \quad as \ t\to +\iy .
\end{align}
Moreover, we obtain in \eqref{U0:scatter:phi+} the approximation 
 $U(t)\phi (x)\approx \frac{1}{\sqrt{2it}} e^{i \frac{x^2}{4t}} \widehat{\phi_+}(\frac{x}{2t})$  in $L^2$  as $t\to +\iy$.
This refines  the result in \cite{GH98} concerning the asymptotic behavior of the linear flow $\phi\mapsto \psi=U(t)\phi$. 
 Intuitively, it implies that the dynamics of the perturbed operator can be divided into two well-understood parts: Scattering states traveling to infinity in a way similar to the free evolution, and bound states which remain confined in a certain sense for all times.  
In the literature, one dimensional scattering results   can be found in  
\cite{BEKS94,Ch01,DT79,Herc89,Tip90} for selfadjoint hamiltonians, 
and \cite{GlowKu2020} for non-selfadjoint ones. 
Scattering problems in multi-dimensions have been considered in e.g., \cite{Ag75,AS71,Enss78, 
Kuroda67,Nenciu75,RS,Teschl2014b}.  
 The methods employed in the study of the above-mentioned  linear perturbation theory  
have also found applications in  long time asymptotics for nonlinear  problems, see 
\cite{GerPuZhang22} and the references therein.  


The structure of this paper is the following. 
In \S \ref{s:preliminV}-\S \ref{s:perturb:F} we  briefly review the spectral theory based on the scattering eigenfunctions 
(technically Jost functions)  
for $H=H_V:=H_0+V$, where  $V$ is a real Borel measure satisfying (\ref{eV:wei2}). 
In  \S \ref{s:L2}-\S \ref{s:main:proof}, using resolvent  calculus 
we establish the $L^2$ theory for
the perturbed Fourier transform $\cal{F}$ associated with $H=H_V$, where $V$ is in $L^1\cap L^2$
so that $H$ has the operator domain $\cD(H)=W^{2,2}$. 
We follow the ideas of  \cite{Ikebe60}, but  also use some
simplifications as given in \cite[vol. III]{RS} and  \cite{Sim71}. 
The work of Ikebe is based on computing the spectral measure for the positive spectrum of $H$, 
in terms of solutions of the Lippmann-Schwinger  equation. This will be a main line of proof in our approach.   

Time-dependent scatting theory is conventionally based on 
the comparison between one-parameter unitary groups $U(t)$ and $U_0(t)$, i.e., the perturbed and unperturbed
evolutions for the hamiltonians. Regarding the development of perturbed Fourier transforms, 
 the  proofs in \cite{GH98} and our work both rely on the existence and completeness of 
the wave operators \eqref{e:Omega-f}. 
However, when $V$ is in $L^1\cap L^2$, our spectral calculus approach yields  some simpler and more  concrete formulae.  
For example,  equation (\ref{surjection_equation}) shows the relation
$\cF f=\cF_0{\Omega^* f}  $  
 between  $\cF$,  $\cF_0$ and the wave operator.   
 The identity (\ref{L2:Pf:Ff}) reveals the exact spectral correlation between 
$\cF$ and the spectral projection $P$ over any interval in the positive half line, 
which  seems new  comparing  \cite{GH98}.  





The remaining of the paper is organized as follows. In section \ref{s:preliminV} we give a brief review on the spectral theory for $H$, where $V$ is assumed as (\ref{eV:wei2}).
Some more elaborations are given pertaining to  the method used to construct scattering eigenfunctions. 
In sections \ref{s:jost} and \ref{s:perturb:F} we recall the definitions of the Jost functions and scattering solutions for $H$.
Then we define the perturbed Fourier transforms $\Fo_\pm$  as  \cite{GH98}. 
In sections \ref{s:L2} to \ref{s:main:proof}, assuming $V\in L^1\cap L^2$,  
we prove Theorem \ref{th:Plancherel} based on the construction of the solutions to (\ref{LS_eq}).   
In section \ref{s:waveOp:scatter} we conclude the paper by showing a general scattering result for $V$ satisfying either (\ref{eV:wei2}) or $V\in L^1\cap L^2$
(Theorem \ref{t:wave:scatter} and Theorem \ref{th:Omega}) 
for the linear evolution operator $U(t)$. 

\section{Preliminaries on the Schr\"odinger operator}\label{s:preliminV}

Suppose $H = - \De +V$ is a Schr\"{o}dinger operator on the real line, where $V$
is a real-valued finite measure  satisfying the short-range condition (\ref{eV:wei2}).
$H$ is the  hamiltonian in the corresponding time-dependent
Schr\"{o}dinger equation \eqref{eU:H_V}, 
 which describes the evolution of quantum wave  packets under the action of $H$. 
The solution is uniquely determined by the initial state $\psi(t,x)=e^{-i t H }\phi(x)$,  \mbox{$\;t\in \R$.}   
 Scattering theory deals  with the wave function $\psi$, whose behavior in the past and in the
future are governed by the  equation (\ref{eU:H_V}).

If assuming $\psi_E (t,x)= u_E(x)e^{-iEt}$ {in (4)}, where \mbox{$E$} is the energy, we arrive at  the time-independent Schr\"odinger equation 
\begin{equation}\label{eU:VE}
 (-\De +V) u_E(x) = E u_E(x)\,
\end{equation}
 which is a typical model for the stationary  waves $u_E$ in an inhomogeneous medium. 
An interpretation is that  $H_0=-\De=-d^2/dx^2 $ describes the behavior of a free particle, 
while $H_0 +V$ describes one particle interacting with an \mbox{external} field of forces $V(x)$. 
In scattering experiment, the  potential $V=V_+-V_-$ is composed of repulsive and attractive components $V_\pm\ge 0$. The negative part $V_-$ may introduce the possibility of localized disturbances, called {\em bound states}, in addition to the scattering solutions of (\ref{eU:VE}).

It is well-known that under the condition (\ref{eV:wei2}) on $V$, a finite measure, 
   $H$ is defined as a selfadjoint operator by quadratic form with form domain $Q(H)$  \cite{Bra85,Herc89,BEKS94,GH98}, 
also  consult \cite{Sim71,Sim82}
 and  \cite[III, p.499]{RS}, \cite{Ch01,Kato66} for the case where $V$ are locally integrable functions.  
The form operator has spectral resolutions $H= \int \lam dE_\lam$ in $\sH=L^2$, where $dE_\lam$ is the associated spectral measure.   
Let $\sigma(H)$ denote the spectrum of $H$.  Then $\sigma(H)\subset \R$ consists of $\sigma(H)=\sigma_{ac}(H)\cup \sigma_{pp}(H)$ with  the singular continuous spectrum 
$\sigma_{sc}(H)$ being empty.  
There holds the {\em spectral decomposition} 
  $ \sH= \sH_{ac} \oplus \sH_{sc} \oplus \sH_{pp}$, where $\sH_{ac}$  denotes the subspace of absolutely continuous vectors, 
  $\sH_{sc}$ the subspace of singular continuous vectors, and $\sH_{pp}$ the subspace of eigen-vectors of $H$.


\subsection{Finite measures and their topology} 
Let $\cD=C_0^\iy(\R)$ be the space of test functions,
that is, the class of $C^\iy$-smooth functions with compact support. 
Let $\cD'$ be the dual space of $\cD$, which is called the space of {\em distributions}. 
The weak topology on $\cD'$ is given by the convergence such that $f_n \to f$
in $\cD'$ is defined as: $\la f_n, \phi\ra\to \la f, \phi\ra$ for all $\phi$ in $\cD$.
Let  $C_b^0=C_b^0(\R)$ be the Banach space of continuous bounded functions  endowed with the  norm $\norm{\cdot}_\iy$.
Denote $\cM$ the dual of $C_b^0$, that is, the space of 
finite 
measures 
on $\R$. By convention, if $\mu\in \cM$, then
the pairing on $\cM\times C^0_b$ is given by
\[
\la \mu, \phi\ra=\int \phi(x) d\mu=\int \phi(x) \mu(dx),\qquad \forall \phi\in C_b^0
\]
where the integral is evaluated as the Lebesgue integral of $\phi$ with respect to  
$\mu$ in $\cM\subset\cD'$.

\subsection{The hamiltonian $H:=H_V$ defined as quadratic form}
Let  $W^{s,2}:=W^{s,2}(\R)$  be the space of distributions  
$W^{s,2}=\{f\in L^2:  D^s f\in L^2\} $ endowed with the norm 
$\norm{f}_{W^{s,2}}= \left(\int (1+ |\xi|^2)^{s} |\hat{f}|^2 d\xi\right)^{1/2} $, 
where $D^s f$ is the $s^{\textup{th}}$-weak derivative of $f$ in $\cD'$ and $\hat{f}$ is the usual Fourier transform 
\begin{equation}\label{F0:fourier}
\hat{f}(\xi) = \frac1{\sqrt{2\pi}} \int_{-\infty}^\infty f(x) e^{-ix\xi}dx\,.  
\end{equation}
Let $V=d\mu_V=V(dx)$ be a real-valued finite Borel measure in $\cM$.  Let $H_0$
be the selfadjoint extension of $-\De=-d^2/dx^2$ in $L^2$.
 The operator sum $H=H_V:=H_0+V$ is in general not defined on $\cD(H_0)=W^{2,2}$.
One can define $H$ in the sense of quadratic forms 
\begin{align}\label{quadratic:H_V}
\la  H_Vf, \phi\ra=\int f'(x){\phi'(x)}   dx+
\int  f(x){\phi}\, V(dx) 
\end{align}
for all $\phi\in W^{1,2}$ and $f\in Q(H)$ \cite{Bra85,Sim71}. 
The form domain $Q(H)$ consists of all functions $f$ in $W^{1,2}$ 
such that the distribution 
$H_Vf=(-\De +V) f\in L^2\cap W^{-1,2}=L^2$. 
The eigenfunctions of $H$ are defined as the set of $f$ in $W^{1,2}$  for which 
\begin{align}\label{eigenfunction:Hv}
(-\De +V -\lam) f=0\quad \text{for some  $\lam\in\R$}\,. 
\end{align}
Note that if $V\in L^2+L^\iy$, then the form domain coincides with the operator domain 
$Q(H)=\cD(H)=W^{2,2}$. Accordingly, the eigenfunctions of $H$ belong {to} $W^{2,2}$. 

\subsection{Green's function and Jost functions}  
Assume $V$ is a  finite measure verifying (\ref{eV:wei2}).     
The Green's function for $H=H_V$ is a fundamental solution of  
\begin{align*}
 (-\De +V-\zeta) G(\cdot, y,\zeta)=\delta(\cdot-y)
\end{align*} 
in the  distributional sense.  Let $R(\zeta)=(H-\zeta)^{-1}$ be the resolvent of $H$.
Let $\zeta=\kappa^2$, $\Im \kappa>0$. 
Then its kernel, the Green's function, is given by 
\begin{align}\label{green:Hv}
G(x,y,\zeta)= \frac{J_\kappa^+(\max(x,y))J_\kappa^-(\min(x,y))}{W(J_\kappa^+,J_\kappa^-)}
\end{align}
 where  $J_\kappa^\pm$ are {\em Jost functions}, see \S \ref{s:jost} or \cite[Theorem 7.1]{GH98}.  
 $G(x,y,\zeta)$ is also analytic on  the negative real axis with only finite number of poles that are zeros of the
  Wronskian $W(J_\kappa^+,J_\kappa^-)=-2i\kappa\al_\kappa$, see  (\ref{W:alpha(z)}). 

\subsection{Scattering solutions} 
Assume $V$ is a real Borel measure  satisfies the short-range condition  \eqref{eV:wei2}. 
 In order to define the Fourier transforms $\Fo_\pm$ associated with $H$ one seeks 
the generalized eigenfunctions of (\ref{S_eq}), or in the complex domain  
\begin{align}\label{eU:Vphi}
(\De +z^2)\vphi^\pm_z = V\vphi^\pm_z \qquad \text{in  $ \cD'(\R)$}
\end{align}  
with the asymptotics
\begin{align}\label{asymp:vphi:z} 
&\vphi^\pm_z(x) \rightarrow \left\{
\begin{array}{ll} 
t_z e^{\pm izx} & x \rightarrow \pm\infty\\
 e^{\pm izx} + r_z^\pm e^{\mp izx} & x\rightarrow \mp\infty
\end{array}\right. 
\end{align} 
where  $t_z$ and $r_z^\pm$ are the transmission and reflection coefficients (\ref{e:tzV}), (\ref{rz:V}). 
In the physics domain $\Im z\ge 0$, $\vphi^\pm_z(x)$ are interpreted as {\em incoming waves} from the left ($x\sim -\iy$) and from the right respectively $(x\sim +\iy)$. 

Under the assumption (\ref{eV:wei2}) on the potential,
Gu\'erin and  Holschneider \cite{GH98} 
were able to solve (\ref{eU:Vphi}) and define the {\em scattering solutions} $\vphi^\pm_z(x)= t_z J_z^\pm(x)$ 
by developing the tools of Jost functions based on the works in  \cite{Bra85,DT79,Herc89,Tip90}. 
This allows to   derive a formula for the spectral measure 
and  compute the time evolution generated by  $H$.  
For $\Im z\ge 0$, $\vphi_z^\pm$ 
also verify the integral equation 
\begin{align}\label{phi:V:scatter}
\vphi_z^\pm(x)= e^{\pm izx} +\frac{1}{2iz} \int e^{iz|x-y|} \vphi^\pm_z(y)V(dy)   
\end{align}
as {\em scattering solutions} of $H$  
that are continuous   in $x,z$ 
and locally  bounded in $x,k$ if $z=k$ are real, see Proposition \ref{tr:vphi_z}. 
Consequently, the perturbed Fourier transforms $\Fo_\pm$ and their adjoints $\Fo^*_\pm$
can be defined by (\ref{Fo+:f})-(\ref{eFo*-g}) using the solutions $\psi^\pm(x,k)$ to the Lippmann-Schwinger equation (\ref{psiV:L-S}), 
 where the functions $\psi^\pm(x,k)$ are defined in (\ref{psi(k)V:L-S}) via  $\vphi_k^\pm(x)=\vphi^\pm(x,k)$.

The wave operators $W_{\pm}=s {-\lim}_{t\rightarrow \pm\infty} e^{itH} e^{-itH_0}$ is said to 
exist and be complete ({\em asymptotically complete}) means that  $W_\pm$  exist as strong limits in 
$\sH=L^2$ and $\ran{W_+}=\ran{W_-}=\sH_{ac}$\,. 
As such, the existence and completeness of $W_\pm$ were {proved} in \cite{Herc89}. 
From these one can show $W_\mp=\Fo_\pm^* \Fo_0$ in (\ref{eWav:pm}), where $\Fo_0$ is the unperturbed Fourier transform (\ref{F0:fourier}),
and derive an intertwining identity (\ref{eW:intertwine}) for spectral operators. 
In  Proposition \ref{HV:spec-measure}  
 we also recall  from \cite
{GH98} the  construction of the spectral measure for $H$ using $\vphi^\pm_k$. 
 The proof is based on Stone's formula (\ref{Eab:R(zeta)})  and  Green's function (\ref{green:Hv}). 

In sections \ref{s:L2} to \ref{s:main:proof},  assuming $V\in L^1\cap L^2$ we   prove the main theorem, namely,  Theorem \ref{th:Plancherel},  
to establish the $L^2$ theory via the $\cF$-transform and 
the spectral projections $P_{[a,b]} $ using the scattering eigenfunctions $e(x,\xi)$ of (\ref{LS_eq}).  
Here we summarize the comparison between this paper and \cite{GH98}  
  in  proving the $L^2$ theory for the perturbed Fourier transforms.  
\begin{enumerate}
\item[(a)] Our approach is mainly  functional calculus theoretical in solving for the continuum eigenfunctions $e(x,\xi)$, 
while  on the spectral domain, the approach in \cite{GH98} is more concentrated  on the analytical properties of $\vphi^\pm(x,z)=J_z^\pm(x)/\al_z$
defined by the Jost functions  (Propositoin \ref{tr:vphi_z}),
where  the values $z=k^2$, $k\in \R_*:=\R\setminus \{0\}$ correspond to the continuum spectrum, and
the eigenvalues in the pure point spectrum 
are the zeros of $\al_z=\al(z)$, or 
the poles of the transmission coefficients $t_z=t(z)$. 
Since our proof uses the Fredholm alternative theorem in Lemma \ref{lem:Fredholm},  
it gives rise to the ``exceptional set'' $\cE_0$ for the $\xi$ variable in Lemma \ref{l:e(xxi):E0}. 
\item[(b)] Both proofs require the existence and completeness of  the wave operators. 
It is worth mentioning that  $L^1\cap L^2$ includes the class $(1+|x|)^{-\frac12-\veps}L^2$.
The latter is a class of Agmon potentials and hence the wave operators are complete and the singular spectra are absent \cite[Theorem XIII.33]{RS}.

\item[(c)] Let $L^1_2:=\{\text{$f$ real-valued}: \int(1+x^2) |f| dx<\iy\}$. Then the class $L^1\cap L^2$ is not included in  $L^1_2$, 
e.g., the function $(1+|x|)^{-1-\veps}\in L^1\cap L^2$ but not in $ L^1_2$.  
If introducing the space $\mathcal{V}_0:=\{f\in L^1_2\cap L^2: \text{$f$ is real-valued}\}$, 
then, roughly speaking, the condition $V\in L^2$ concerns  the local integrability while $L^1_2$ concerns the weighted integrability near infinity. 
\item[(d)] Suppose $V\in \mathcal{V}_0$. 
Since any solution to  the integral equation (\ref{psiV:L-S}) 
must satisfy (\ref{eU:Vphi}) along with the asymptotics (\ref{a:psi:k+})-(\ref{a:psi:k-}),    
 $e(x,\xi)$ and $\psi^+(x,\xi)$ coincide 
 in view of (\ref{LS_eq}) and (\ref{psiV:L-S}).  
 It follows that $\cF=\Fo_+$. 
\item[(e)] The  eigenfunctions $e(x,\xi)$ can be used to compute the time 
evolution $U(t)\phi$  in light of  Theorem \ref{th:Plancherel}. 

\medskip
Let $V\in \mathcal{V}_0$.  
 Then $e(x,\xi)=\vphi^+(x,\xi)$, $\cF=\Fo_+$ and the set $\cE_0=\emptyset$ in Theorem \ref{th:Plancherel}. 
Note that for any bounded function $\vphi$ the  operator $\vphi(H)$ can be defined by the perturbed Fourier transform in virtue of \eqref{Eac:F*F}, 
\eqref{F*Fo:inversion} and 
the uniqueness of
the spectral theorem. 
Thus, for all $\phi\in \sH_{ac}$, we have 
 \begin{align}\label{eU(t):cF} 
 U(t)\phi(x)= \cF^* e^{-it \xi^2}\cF \phi(x) 
= \int_{\R^2}  e^{-it \xi^2} e(x,\xi) \bar{e}(y,\xi) \phi(y) d\xi dy 
\end{align} 
which is more concrete than the expression for $e^{-itH}$ in \cite[IX]{GH98}.   
\end{enumerate}  

\section{Jost functions and scattering solutions}\label{s:jost} 
Jost functions are  fundamental and useful tools in the study of spectral properties for $H=H_V$.  
Let $V$ be a real-valued short-range potential verifying (\ref{eV:wei2}).    
For $\Im z\ge 0$ the Jost functions $J_z^\pm $ are defined as the solutions of  
\begin{align}
&J^+_z(x)= e^{izx} + \int_x^\infty \frac{\sin (z(y-x))}{z}J^+_z(y)
V(dy)\label{eJ:+}\\
&J^-_z(x)= e^{-izx} + \int^x_{-\infty} \frac{\sin (z(x-y))}{z}J^-_z(y) V(dy). \label{eJ:-}
\end{align}
Then, we have the following properties. 
\begin{enumerate}
\item[(i)] 
$J^\pm_z$ solve the Helmholtz equation (\ref{eU:Vphi}).
\item[(ii)] There holds the symmetry $\overline{J^\pm_k}(x)=J^\pm_{-k}(x)$, $k\in \R_*$. 
\item[(iii)] The value of the Wronskian 
\begin{equation}\label{W:alpha(z)}
W[J_z^+, J_z^-](x)= -2iz \al_z\,. 
\end{equation} 
\item[(iv)]   The asymptotics hold 
\begin{equation}\label{eJz:asymp}
J_z^\pm(x) =\left\{
\begin{array}{ll}
e^{\pm izx}+O(x^{-2} e^{\mp i zx}), & x\rightarrow \pm\infty\\
\alpha_z^\pm e^{\pm izx} + \beta_z^\pm e^{\mp izx} +O(x^{-2} e^{\pm i z x } ), &
x\rightarrow \mp\infty\,.
\end{array}\right. 
\end{equation}
Here, for all $\Im z\ge 0$,  
 $\al_z:=\al_z^+=\al_z^-$ are given by  
\begin{align}
&\alpha_z^\pm =1-\frac{1}{2iz} \int j^\pm_z(y)V(dy), \,  \label{al(z):V}\\
&\beta_z^\pm = \frac1{2iz} \int e^{\pm 2 i zy} j_z^\pm(y)V(dy)\label{e:beta(z)} 
\end{align} 
and $j_z^\pm(x)=e^{\mp izx}J_z^\pm(x)$ are the modified Jost functions verifying 
\begin{equation}\label{jz:asym}
j_z^\pm(x) =\left\{
\begin{array}{ll}
1+O(x^{-2} e^{\mp 2i zx}), & x\rightarrow \pm\infty\\
\alpha_z^\pm  + \beta_z^\pm e^{\mp 2i zx} +O(x^{-2} ), & x\to \mp\infty\,.
\end{array}\right. 
\end{equation}
In addition, for all $k$ in $\R$ 
\begin{equation*}
\beta_k^+= -\overline{\beta_k^-},\quad  |\al_k|^2-|\beta_k^\pm|^2 = 1  .   
\end{equation*}
\end{enumerate}

The scattering solutions are defined as 
  $\vphi_z^\pm= J_z^\pm/\alpha_z$  for $\alpha_z \neq 0$. 
It follows that  for ${\Im} z\ge 0$, $\vphi_z^\pm$ solves equation \eqref{eU:Vphi}
  with the asymptotics (\ref{asymp:vphi:z}), which  is equivalent to  
the scattering  equation (\ref{phi:V:scatter}). 
This requires   
\begin{align}
& t_z=t_z^\pm = 1+\frac{1}{2iz} \int e^{\mp izy}\vphi_z^\pm(y)V(dy) \label{e:tzV}\\
&r_z^\pm= \frac{1}{2iz} \int e^{\pm izy}\vphi_z^\pm(y)V(dy)\,   \label{rz:V}
\end{align}
such that 
\begin{align}\label{t(z):r(z):alpha:beta}
&t_z =\alpha_z^{-1}\   and \  r_z^\pm= \beta_z^\pm/\alpha_z\,.
\end{align}

 In order to solve the  integral equation (\ref{phi:V:scatter}), a traditional procedure is to find $J_z^\pm$
using the modified Jost function $j_z$, the latter can be found using a series expansion based on suitable estimation of
iterated integrals. Then one defines $\vphi_z^\pm(x,z):=\vphi_z^\pm(x)= J_z^\pm(x)/\alpha_z$ as above.  
We recall from \cite[Theorem 6.1]{GH98} the following asymptotics and spectral properties for $\vphi_z^\pm$.  
\begin{prop}\label{tr:vphi_z}  Let $V$ satisfy the condition (\ref{eV:wei2}). Let $\Im z\ge 0$ and $\al_z\ne 0$. Then $\vphi_z^{\pm}$ have the following  properties.
\begin{enumerate}
\item[(a)] The functions $\vphi_z^\pm$ are solutions of (\ref{phi:V:scatter}). 
Equivalently, $\vphi_z^\pm$ solve (\ref{eU:Vphi}) with the  
assymptotics (\ref{asymp:vphi:z})  as $|x|\to\iy$, where
  $t_z =\alpha_z^{-1}$ and $r_z^\pm= \beta_z^\pm/\alpha_z$ as given in  (\ref{t(z):r(z):alpha:beta}). 
\item[(b)] If $V$ has compact support, then the above asymptotics (\ref{asymp:vphi:z}) are exact outside the support. 
\item[(c)] 
For fixed $x$, $\vphi_z^{\pm}(x)$ are meromorphic in ${\cal Im}z >0$,
with a {\em finite} number of poles located on the imaginary axis,
corresponding to the zeros of $\alpha_z$. Moreover, the functions
$\vphi_z^{\pm}(x)$ are continuous on the closed upper half-plane ${\Im}z \ge 0$ and continuously differentiable on ${\Im}z \ge 0,
z\neq 0$.  The same analyticities hold for the transmission and reflection coefficients. 
\item[(d)] On the real $k$ axis, $(\min \{ \frac1{|k|}, 1+|x|\} )^{-1}\vphi_k^{\pm}(x)$ and
$(1+|x|)^{-1} {\dot{\vphi}}_k^{\pm}(x)$ with respect to $k$ are 
uniformly bounded.  In addition,  
\begin{align}\label{ePhi:k}
 \overline{\vphi_k^{\pm}  }=\vphi_{-k}^{\pm}, \;\; \bar{t}_k= t_{-k},\;\; 
\overline{ r_k^{\pm }} = r_{-k}^\pm
 \end{align} 
 which implies for $\al_k\ne 0$   
 \begin{align}
 & \overline{\al_k}= \al_{-k},\;\; 
\overline{ \beta_k^\pm } = \beta_{-k}^\pm
 \end{align}
and \begin{align} 
r_k^+ \bar{t}_k + t_k \overline{r_k^-} =0, \;\;  |t_k|^2 +|r_k^\pm |^2=1.\label{rt:k}
\end{align} 
\end{enumerate} 
\end{prop} 
\begin{remark}\label{error:GH}  
There was an error in Theorem 6.1 \cite{GH98} about the  uniform boundedness for $\vphi^\pm(x)$. 
The correct statement should be part (d) in the  proposition above. 
More precisely, one can  calculate  for all $\Im z\ge 0$ and all $x\in\R$: Let $z=\al+ i \eta$, $\eta> 0$, then for $x<0$ 
\begin{align*} 
&|e^{-i zx } \frac{\sin zx}{z} |= \frac{e^{\eta x}}{|z|}\sqrt{\ch^2(\eta x)-\cos^2(\al x)} \le \frac1{ |z|}\,.
\end{align*}   
 If $z=k \in \R_*$, then 
\begin{align}  
&|e^{-i zx } \frac{\sin zx}{z} |= \frac1{|z|}\vert\sin z x\vert 
\le  \begin{cases}  |x| &  if \ |zx|<1\notag\\
\frac1{|z|}&  if \ |zx|\ge 1 .
\end{cases}
\end{align} 
This implies that  in \cite[Theorem 5.1]{GH98}, it should read
$z\mapsto (1-G_z^\pm)\inv$ is analytic in $\Im z>0$ and continuous in $\{z:\Im z\ge 0, z\ne 0\}$. 
In addition, equation (5.10) therein should have a linear bound by $|x|$ 
 for the first integral  and a quadratic bound by $x^2$ for the second integral in the case $x$ being negative. 
In all scenario, we have 
\begin{align} 
& |\vphi^\pm_z( x)  |\le  C\min \{ 1+| x|, \frac1{|z|} \}  \, \qquad \forall x\in\R,  \Im z\ge 0, z\ne 0\,. 
\end{align}
See Example \ref{ex:deltaV} where $V$ is a multiple of the dirac measure $\de$.  
\end{remark}  

\begin{remark}\label{re:simple:eigen}  
   The pure point spectrum of $H$ consists of a finite set of negative numbers,
   which precisely  correspond to the (simple) zeros of $\al(z)$. 
Namely,    if $\al(z_j)=0$, then $z_j=i\kappa_j$, $\kappa_j>0$ for $j=1,\dots,J_0$ 
and $\lam_j:=z_j^2=-\kappa_j^2$  are the simple eigenvalues for $H=H_V$. 
This can be seen from either the asymptotics 
 (\ref{asymp:vphi:z})  for $\vphi_z^\pm(x)$,
or, the asymptotics (\ref{jz:asym}) for $j_z^\pm(x)$. Indeed, for the existence of bound state (outgoing),
it is necessary $\al(z)$ vanishes.
For $\lam_j= -\kappa_j^2$,  $j=1,\dots, J_0$, 
the associated eigenstates $\phi_j$ are solutions of (\ref{u:eigenstate}) in $Q(H)\subset W^{1,2}$
\begin{align}
H_V u= \lam_j u \label{u:eigenstate}
\end{align}
given by  $u(x)=\phi_j(x)=O(e^{ -\kappa_j |x|})$    having exponential decay as $|x|\to \iy$.
\end{remark}  

\begin{prop}[Spectral measure for $H_V$]\label{HV:spec-measure} Let $V$ satisfy the short-range condition (\ref{eV:wei2}).  
Then the positive spectrum of $H_V$ is purely absolutely continuous and the negative spectrum is composed of a finite number of negative eigenvalues 
\begin{align*}
\sigma_{ac}(H)=[0,\iy), \  \sigma_{sc}(H)=\emptyset, \ \sigma_{pp}(H)=\{-\kappa_1^2, \dots, -\kappa_{J_0}^2\} 
\end{align*}
where $i\kappa_1, \dots, i\kappa_{J_0}$ are the simple zeros of $\al(z)$ in the upper half plane $\Im z>0$.
Let $\{E_\lam\}$ denote the spectral resolution to $H=H_V=\int \lam dE_\lam$.  For any states $f,g$ in $\cD$ the associated spectral measure 
$\mu_{f,g}$ given by $d\mu_{f,g}(\lam)= d_\lam(E_\lam f, g)$  is absolutely continuous on the positive axis 
with density  
\begin{align*} 
\frac{d\mu_{f,g}(\lam)}{d\lam}=m_{f,g}(\lam)= \int_{\R^2}  \overline{g(y)} f(x) p(x,y,\lam)  dxdy
\end{align*}
where the kernel is given as 
\begin{align}\label{Hv:spectra-measure} 
& p(x, y, \lam)= \frac1{2\pi} \Re (\frac{\vphi^+_k(y)\vphi^-_k(x) }{ k t_k}   )
=  \frac1{2\pi} \Re (\frac{\vphi^+_k(x)\vphi^-_k(y) }{ k t_k}), \qquad \lam=k^2, k>0.
 \end{align}
\end{prop} 

 The spectral density kernel (\ref{Hv:spectra-measure})  in the proposition was shown in \cite{GH98} 
and can be derived from the resolvent $R(\zeta)=(H-\zeta)\inv$ and the analyticity of  its kernel $G(x,y,\zeta)$ in (\ref{green:Hv}). 
In fact, if $[\al,\beta]\subset\R_+$, then Stone's formula on the spectral projections  reads  for  $0<\al<\beta$  
\begin{align}\label{Eab:R(zeta)} 
& P_{(\al,\beta)}= \sslim_{\eps\to 0}  \frac1{2\pi i} \int_\al^\beta ( R(\zeta+i\eps) -R(\zeta-i\eps) ) d\zeta  \,.
 \quad  
 \end{align} 
Therefore, equation (\ref{Hv:spectra-measure}) follows from (\ref{green:Hv}), (\ref{W:alpha(z)}) and (\ref{t(z):r(z):alpha:beta}). 

\bigskip
\begin{example}[P\"oschl-Teller potential \cite{OZh06} ]\label{ex:sech}   
Let $V(x)= -6 \,\sech^2 x $. Then the  fundamental solutions of (\ref{eU:Vphi}) are: For $\Im k\ge 0 $
\begin{align*}
&\vphi^+(x,k)= e^{ikx}\frac{1+k^2+3ik\th x-3\th^2 x }{(k+i)(k+2i)} \\
&\vphi^-(x,k)= e^{-ikx}\frac{1+k^2-3ik\th x-3\th^2 x }{(k+i)(k+2i)} \,.
\end{align*}
 In addition,  
$$
 \al(k)=\frac{(k-i)(k-2i) }{(k+i)(k+2i)}\, 
$$
so that $t_k=\frac{(k+i)(k+2i) }{(k-i)(k-2i)}$ and ${r_k^\pm}=0$.  
If $k=i$, one has the bound state $\vphi^+(x,i)=\frac12 \sech x\tanh x$.  
If $k=2i$,  $\vphi^+(x,2i)=\frac14 \sech^2x$.   
\end{example}


\begin{example}[$\de$ potential]\label{ex:deltaV}
 Let  $V(x)= 2b \delta (x)  $ with $b\in\R$.  Then the scattering solutions of 
$ u''+[k^2- 2b\delta(x)  ]u =0 $ are: For $\Im k\ge 0$
$$
\vphi^+(x,k)=\left\{
\begin{array}{ll}
e^{ikx}, & x\ge 0\\
e^{ikx}-\frac{2b}{k}\sin kx, & x<0
\end{array}\right.
$$

$$
\vphi^-(x,k)= \left\{
\begin{array}{ll}
e^{-ikx}+ \frac{2b}{k}\sin kx , & x\ge 0\\
e^{-ikx}, & x<0\,.
\end{array}\right.
$$
We have  $t_k=\frac{k }{ k+i b}$ and $ r^\pm_k= \frac{-ib}{ k+ib } $. 

There are {\em bound state} solutions  if $b<0$. 
Indeed 
from  $\al_k=t_k\inv$, 
  we see $ k=-ib $ is the zero of $\al_k$. Then 
$\vphi^+(x,-ib) =e^{b |x|} $,  
which is a localized  bound-state solution. 
\end{example}

Classical solvable examples for $H_V$ can be found in e.g., \cite{PandeyPathak78, Ti46} 
where singular boundary conditions are present.
Eigenfunction expansions for pair of adjoint differential operators reducible to the form  
$L:=\frac{d^2}{dx^2}+b(x)\frac{d}{dx}+c(x)$
and $L^*:=\frac{d}{dx}( \frac{d}{dx}-b(x)) +c(x)$ on an interval 
with singular coefficients $b$  were studied earlier in \cite{Elliott55}.
Further concrete models involving nonlocal singular interactions have been considered in \cite{GlowKu2020}. 

\section{Perturbed Fourier transforms $\Fo$ and $\Fo^*$}\label{s:perturb:F}  
Throughout this section  assume $V\in \cM_2$ is real, i.e.,  $V$ a real-valued finite measure satisfying (\ref{eV:wei2}).  
Following the conventional notion, one can define the perturbed Fourier transforms 
 (called generalized Fourier transforms in \cite{GH98}).   For $k$ in $\R$ define the functions $\psi_k^\pm=\psi^\pm(x,k)$ as 
\begin{equation}\label{psi(k)V:L-S}
\begin{aligned}
& \psi_k^+=
\begin{cases} \vphi_k^+& \text{if} \ k\ge 0\\
 \overline{\vphi_k^-}& \text{if} \ k< 0
\end{cases}
\quad 
& \psi_k^-=
\begin{cases} \vphi_k^-& \text{if} \ k\ge 0\\
 \overline{\vphi_k^+}& \text{if} \ k< 0.
\end{cases}
\end{aligned}
\end{equation} 
From (\ref{ePhi:k}) we have if $k \in \R$ 
\begin{align*}
& \psi_k^+=
\begin{cases} \vphi_k^+& \text{if} \ k\ge 0\\
 {\vphi_{-k}^-}& \text{if} \ k< 0 
\end{cases}
\quad 
& \psi_k^-=
\begin{cases} \vphi_k^-& \text{if} \ k\ge 0\\
 \vphi_{-k}^+& \text{if} \ k< 0 .
\end{cases}
\end{align*}
 In general,  $\psi_k^+,  \overline{ \psi_k^-} $ are correlated  but
not  the same even for radial $V=V(|x|)$. So the perturbed $\Fo_+$ and $\Fo_-$  
defined as  (\ref{Fo+:f}) and (\ref{e:Fo-psi}) are not the same, see Examples \ref{ex:sech} and \ref{ex:deltaV}.  

The functions $\psi^\pm_k$ solve the Lippmann-Schwinger type equation (\ref{psiV:L-S}) in $z=k\in \R_*=\R\setminus\{0\}$ with respect to the measure $V$  
\begin{align}\label{psiV:L-S}
&\psi^\pm(x,k)= e^{\pm ikx} +\frac{1}{2i |k|} \int e^{i |k| |x-y|}V(dy)\psi^\pm(y,k) 
\end{align}
which follows from  (\ref{phi:V:scatter}) and the  definitions above. 
It is easy to verify that any solution to \eqref{psiV:L-S} solves (\ref{eU:Vphi}) with $ z=k\in\R_*$
and admits the following asymptotics: If $k>0$, 
\begin{align}\label{a:psi:k+} 
&\psi^+_k(x) \rightarrow \left\{
\begin{array}{ll} 
t_k e^{ ikx} & x \rightarrow +\infty\\
 e^{ ikx} + r_k^+ e^{- ikx} & x\rightarrow -\infty
\end{array}\right. 
\end{align} 
and if $k<0$
\begin{align}\label{a:psi:k-} 
&\psi^+_k(x) \rightarrow \left\{
\begin{array}{ll} 
\overline{t_k} e^{ ikx} & x \rightarrow -\infty\\
 e^{ ikx} + \overline{r_k^-} e^{- ikx} & x\rightarrow +\infty\,.
\end{array}\right. 
\end{align} 
Similar asymptotic behaviors hold for $\psi^-_k(x)$.
Thus, if $k>0$, $\psi^\pm_k(x)$ have the same asmptotics as   $\vphi^\pm_k(x)$;
and if $k<0$, $\psi^\pm_k(x)$ have the same asmptotics as   $\overline{\vphi^\mp}_k(x)$.
We see on the physics domain $\psi_k^+$ represent {\em incoming waves from the left} while 
$\psi_k^-$ {\em incoming waves from the right}.  
Note that the fundamental function of $(-\frac{d^2}{dx^2}-k^2)u=\de(x-y)$ is given by 
\begin{equation}
(-\frac{d^2}{dx^2}-k^2)\inv(x,y)= -\frac{ e^{i |k| |x-y|}}{2i |k|} \,.
\end{equation} 
So,  we can formally solve (\ref{psiV:L-S}) by  
\begin{align}\label{ePsi:R0V}
\psi^\pm(x,k)= (I+R_0(k^2)V)\inv e^{\pm ik x}\,
\end{align}
{where $R_0(k^2)= (H_0-k^2)\inv$}. 
For $f\in L^1$, define the perturbed Fourier transforms associated to $H_V$
\begin{align}
&\Fo_+f (k)=  \frac{1}{\sqrt{2\pi}}\int_\R  f (x) \overline{\psi^+ (x,k)}dx\label{Fo+:f}\\ 
&\Fo^*_+g (x)=  \frac{1}{\sqrt{2\pi}}\int_\R  g (k) {\psi^+}(x,k)dk\label{Fo*+g}\\  
&\Fo_- f (k)=  \frac{1}{\sqrt{2\pi}}\int_\R  f (x) {\psi^-}(x,k)dx\label{e:Fo-psi}\\ 
&\Fo_-^*g (x)=  \frac{1}{\sqrt{2\pi}}\int_\R  g (k) \overline{\psi^-(x,k)}dk\label{eFo*-g}
\end{align}
where $\Fo_\pm$ and $\Fo_\pm^*$ are continuous mappings from $L^2\to L^2$. 
From  Proposition \ref{HV:spec-measure} 
we see the spectral measure  involves  
$\vphi^\pm_k$, so it is natural and necessary to consider $\Fo_\pm$ and $\Fo^*_\pm$  as above.

Recall {the  following from}  
 \cite[sections VIII-IX]{GH98}. 
\begin{prop}\label{p:Fo:L2} The transforms $\Fo_\pm$ and $\Fo_\pm^*$ extend to continuous mappings from $L^2\to L^2$.
Moreover, $\Fo_\pm^*$ are the adjoint of $\Fo_\pm$ satisfying 
\begin{align}\label{F*Fo:inversion}
&\Fo_\pm^* \Fo_\pm= P_{ac}\,, \qquad \Fo_\pm\Fo^*_\pm = I \,.\quad 
\end{align}
That is, $\Fo_\pm:  \sH_{ac}\to L^2$ and $\Fo_\pm^*:  L^2\to \sH_{ac}$ are inverse isometry operators.  
\end{prop}

The Fourier inversion identities  (\ref{F*Fo:inversion}) result from 
the existence and (asymptotic) completeness of $W_\pm$, which  
was proven in \cite{Herc89} for general class of distributional potentials. 
\begin{theorem}\label{Fo:waV} \textup{(\cite[section IX]{GH98})} The wave operators 
\begin{align}\label{eWav:pm}
W_\pm:=\slim_{t\to \pm\iy} e^{itH} e^{-itH_0}=\Fo^*_\mp\Fo_0\,
\end{align}
exist and are complete, that is, 
\begin{align*}
\ran(W_+)=\ran(W_-)=\sH_{ac}\,.  
\end{align*}
\end{theorem}



As a consequence of  Theorem \ref{Fo:waV}, the generalized functions $\psi_k^\pm$ 
constitutes a complete orthonormal  set in $\sH_{ac}$.  Moreover, the equations 
\begin{align}
 W^*_\pm W_\pm=I, \quad W_\pm W_\pm^*=P_{ac}\,. \label{ac:WWstar} 
\end{align}
and  (\ref{F*Fo:inversion}) hold true.  

Equation (\ref{F*Fo:inversion}) allows us to state a spectral theorem for $H$ by defining the spectral operator $\vphi(H)$ in $\sH_{ac}$ for all $\vphi$ in $\cB(\R)$, the space of all bounded Borel measurable functions. 
 For all $f\in \sH_{ac}$\,, define  
\begin{align}\label{spec:F*Fo}
&\vphi(H) f=\Fo_\pm^* \vphi(k^2) \Fo_\pm f \,.
\end{align} 
Then, the operators $\vphi(H)$: $\sH_{ac}\to\sH_{ac}$ coincides with the  (unique)  functional calculus form, consult \cite[Theorem VIII.5]{RS}. 
We have seen that the definitions of the perturbed Fourier transforms are based on the spectral measure representation in Proposition \ref{HV:spec-measure} 
and the asymptotic completeness of the wave operators. 


\section{The $L^2$ theory for $H_V$, $V\in L^1\cap L^2$}\label{s:L2}  
 Throughout this section till section \ref{s:main:proof} we assume $V\in L^1\cap L^2(\R)$ is real-valued. 
We shall establish the $L^2$ theory for the perturbed Fourier 
transform as stated in Theorem \ref{th:Plancherel}. 

Recall that  the operator $-d^{\,2}/dx^2$ on $L^2(\R)$ is essentially
self-adjoint on the domain $\cD=C^\infty_0(\R)$ and  $H_0$ denotes its unique self-adjoint extension. 
The operator $-d^{\,2}/dx^2 + V(x)$ on $L^2(\R)$ is also essentially self-adjoint on the  domain $\cD$. 
Let $H$ denote  its unique self-adjoint extension. The domains $\mathcal{ D}(H_0)$ and 
$\mathcal{ D}(H)$ are the same, namely,  
\begin{align*}
\mathcal{ D}(H_0) = \{ f \in L^2: f^{\prime \prime} \in L^2\ \mbox{in  the sense of distributions} \}. 
\end{align*}   
In fact, if $V\in L^2+L^\infty$, then by Kato-Rellich lemma, $\cD(V)\supset D(H_0)$ and $H=H_V:= H_0+V$ 
is self-adjoint on $\cD(H_0)$. 
If in addition $V\in L^1$, then $\sigma_{sc}(H)=\emptyset$ and hereby $\sH=\sH_{ac}\oplus \sH_{pp}$. 
Here and in what follows, we shall always adopt the same notations  
$\sH_{ac}$, $\sH_{pp}$, $\sH_{sc}$ 
and their corresponding projections $P_{ac}$, $P_{pp}$, $P_{sc}\,$, etc.  as in \S \ref{s:preliminV} to \S \ref{s:perturb:F}.

\begin{theorem} \label{th:Plancherel}
Suppose $V\in L^1\cap L^2$ is real-valued.
 Then there  exists a family of solutions $e(x,\xi)$, a.e. $\xi \in \R$,  to equation   
(\ref{S_eq}) with the following properties.  
\begin{enumerate}
\item[(a)]  If $f \in L^2$, then for every $M > 0$, $\xi \in \R$, the integral
$$
(2\pi)^{-1/2} \int_{-M}^M f(x) \overline{e(x,\xi)} \, dx
$$
converges absolutely. As a function of $\xi$ this integral defines an element
of $L^2$. Furthermore, there exists an element $f^\sharp \in L^2$ such that
\begin{align}\label{cF:L2}
&(2\pi)^{-1/2} \int_{-M}^M f(x) 
\overline{e(x,\xi)} \, dx  \rightarrow {f}^\sharp (\xi) :=\cal{F}f  
\end{align}
in $L^2$ norm as $M \rightarrow \infty$.  
\item[(b)] There exists a countable disjoint collection of intervals 
$[\alpha_i,\beta_i) \subset (0,\infty)\setminus \cal{E}_0$, with $\cE_0$ a bounded   set of measure zero,  
such that for every $f \in L^2$, 
$$
\sum_{i=1}^N (2\pi)^{-1/2} \int_{\alpha_i \leq \xi^2 \leq \beta_i}
{f}^\sharp(\xi)e(x,\xi) \, d\xi \rightarrow P_{ac} f(x)
$$
in $L^2$ norm as $N \rightarrow \infty$. 
\item[(c)]   For all $f \in L^2$, we have 
\begin{align*}
& \Vert P_{ac}f \Vert_{L^2}=\Vert {f}^\sharp \Vert_{L^2}\\
& \Vert f \Vert_{L^2}^2 = \Vert {f}^\sharp\Vert_{L^2}^2  +\sum_{\lam_k\in \sigma_{pp}} |( f,e_k) |^2\,,
\end{align*}
where $\{e_k\}$ is a sequence of orthonormal basis in 
$\sH_{pp}$ such that  $e_k$ are the eigenfunctions of $H$ corresponding to the eigenvalues $\lam_k\in \sigma_{pp}(H)\subset (-\iy,0)$. 
\item[(d)]  If $f \in \mathcal{ D}(H)$  
then ${(Hf)}^\sharp(\xi) = \xi^2 {f}^\sharp(\xi)$  for almost every $\xi^2\in \sigma_{ac}(H)=[0,\iy)$. 
\end{enumerate}\end{theorem} 


\begin{remark} Accordingly, we can define the perturbed Fourier transform in $L^2$
\begin{align}\label{cFf:L2:V}
&\cF f(\xi):=\mylim_{M\to \iy} (2\pi)^{-1/2} \int_{-M}^M f(x) \overline{e(x,\xi)} \, dx  
\end{align}
and the adjoint of $\cF$   by 
\begin{align}\label{eF*:inv}
\cF^*g (x):=\mylim_{N\to \iy}\sum_{i=1}^N (2\pi)^{-1/2} \int_{\alpha_i \leq \xi^2 \leq \beta_i} {g}(\xi)e(x,\xi) \, d\xi \,,
\end{align}
because  $\cF$ is surjective from $\sH_{ac}$ onto $L^2$ in view of Proposition \ref{Th:F-surj}.  
Here the $\mylim$ means the convergence in the $L^2$ sense. 
 Then it follows that 
\begin{align}\label{Eac:F*F}
& P_{ac}=\cF^* \cF ,  \qquad \cF\cF^*=I\,,  
\end{align}   
see  Proposition \ref{c:wave:F*F0}. 
Consequently, for all $f \in \cD(H)$, the following holds in $L^2$:  
\begin{align*}
& Hf= \cF^* \xi^2 \cF{f}+ \sum_k \lam_k ( f, e_k) e_k\,. 
\end{align*}
Moreover, for all $\vphi\in C^0_b(\R)$ one can define the spectral operator $\vphi(H)=\int \vphi(\lam)dE_\lam$ by 
\begin{align}\label{def:specOp}
& \vphi(H)f=\cF^* \vphi(\xi^2)\cF + \sum_k \vphi(\lam_k) ( f, e_k ) e_k\,\qquad\forall f\in L^2\,.  
\end{align} 
\end{remark}  

\begin{remark}  
It is also well-known that the asymptotic completeness of wave operators 
imply that  $\sigma_{sc}(H)$ is empty, see \eqref{e:Omega-f} and
 Lemma \ref{l:exist:ac}. 
\end{remark}

We divide the proof of Theorem \ref{th:Plancherel} into some lemmas and propositions in four sections. 
We begin with solving the scattering eigenfunctions $e(x,\xi)$  to the associated integral equation
\begin{equation}\label{LS_eq}
e(x,\xi) = e^{ix\xi} + 
(2i \vert \xi \vert)^{-1} \int_{-\infty}^\infty e^{i\vert \xi \vert  
\vert x-y \vert} V(y) e(y,\xi) \, dy.
\end{equation} 

\begin{remark}\label{re:E(xxi)pm:LS} We would like to mention that Theorem \ref{th:Plancherel} 
holds true for $e_\pm(x,\xi)$ that solve the following equation 
\begin{equation}\label{E_pm(xxi):L-S}
e_\pm(x,\xi) = e^{\pm ix\xi} + 
(2i \vert \xi \vert)^{-1} \int_{-\infty}^\infty e^{i\vert \xi \vert  \vert x-y \vert} V(y) e_\pm(y,\xi) \, dy.
\end{equation} 
Accordingly, one may define the transforms $\cF_\pm$ and $\cF_\pm^*$  in the pattern of Theorem \ref{th:Plancherel} 
but in the form of (\ref{Fo+:f})-(\ref{eFo*-g}).  
With this note, all the relevant statements and equations are valid that involve $\cF_\pm$, 
see equation (\ref{surjection_equation}), Proposition \ref{c:wave:F*F0} and Theorem \ref{th:Omega}.  
For a fixed idea, our proofs will be primarily concentrated on the case where $e(x,\xi)=e_+(x,\xi)$. 

\bigskip
Note that $(d^{\,2}/dx^2 + \xi^2) 
(2i \vert \xi \vert)^{-1} e^{i\vert \xi \vert \vert x-y \vert} = 
\delta(x-y)$. So, at least on a formal level, a solution to (\ref{LS_eq}) is  
a solution to (\ref{S_eq}). More precisely, it is easy to check that 
if $e(x,\xi)$ is a 
solution to (\ref{LS_eq}) such that $V(\cdot)e(\cdot,\xi) \in 
L^1(\R)$, then this solution is a {\em weak solution} to equation (\ref{S_eq}).  

It is not immediately evident that (\ref{LS_eq}) has solutions. Consider 
however the modified equation
\begin{equation}
\psi(x,\xi) = \vert V(x) \vert^{1/2} e^{ix\xi} + (2i \vert \xi \vert)^{-1} 
\int_{-\infty}^\infty \vert V(x) \vert^{1/2} e^{i\vert \xi \vert 
\vert x-y \vert}
V^{1/2}(y) \psi(y,\xi) \, dy.
\label{LS_mod_eq}
\end{equation} 
(Here $V^{1/2}(x)$ is notation for $\vert V(x) \vert^{1/2} {\rm sgn}\ V(x)$.) 
\end{remark} 

\begin{lemma}\label{l:e(xxi):E0} Let $V\in L^1$.  
There exists a bounded set $\mathcal{ E}_0 \subset (0,\infty)$ of measure zero
such that 
 equation  (\ref{LS_mod_eq}) has a unique solution $\psi(\cdot,\xi) \in L^2$ for all
$\xi \in \R $ with $\vert \xi \vert \in (0,\infty) \setminus \mathcal{ E}_0$.  
If $\vert \xi \vert \in (0,\infty) \setminus \mathcal{ E}_0$, then the function  
$$
e(x,\xi) = e^{ix\xi} + (2i\vert \xi \vert)^{-1} 
\int_{-\infty}^\infty e^{i\vert \xi \vert \vert 
x-y \vert} V^{1/2}(y) \psi(y,\xi) \, dy 
$$
exists as a bounded (locally) absolutely continuous function
($C^\infty$ function of $x$ when $V$ is $C^{\infty}$) and uniquely
obeys equation (\ref{LS_eq}).
\end{lemma}

\begin{proof}
For each $w \in U:=\{ \mathcal{ I}m\, w > 0\} $, let $T(w):= |V|^{1/2} (H_0-
w^2)^{-1} V^{1/2}$. Since $V\in L^1  $, $T(w)$ is a bounded 
linear operator on $L^2$ with integral kernel 
$$
T_w (x,y) = - (2iw)^{-1} \vert V(x) \vert^{1/2} e^{iw \vert x-y \vert}
V^{1/2}(y).
$$
Clearly $T : U \rightarrow {\cal B}(L^2)$ is an analytic
function of $w$ and has continuous extension to 
$\partial U\setminus \{0\}$.
For each $w\in \overline{U}\setminus \{0\}$, $T(w)$ is a Hilbert-Schmidt operator.

Equation
(\ref{LS_mod_eq}) has the form $(I + T(\vert \xi \vert))a = b$, 
where $b \in L^2$ denotes
the function $\vert V(x) \vert^{1/2} e^{ix\xi}$ and $a$ represents the 
desired solution $\psi(x,\xi)$. 
Because of the factor $w^{-1}$ in the definition of $T_w(x,y)$ 
 there exists some $M > 0$ such that $T(w)$
is a contraction on $L^2$ for all $w \in \C$ with ${\cal I}m\ w \geq 0$, 
$\vert w \vert > M$.  
So, the operator 
$(I + T(w))^{-1} \in {\cal B}(L^2)$ 
exists whenever ${\cal I}m\ w \geq 0$, $\vert w \vert > M$. 

By a version of the analytical Fredholm theorem (Lemma \ref{lem:Fredholm}) there exists a bounded subset of measure zero
$S \subset \R \setminus \{ 0 \}$ such that  
the operator 
$(I + T(w))^{-1} \in {\cal B}(L^2)$ exists  for all $w \notin S$ with $S\cup \{0\} $ being closed. This proves the first statement in 
the lemma, if we take ${\cal E}_0$ \emph{to be the set  of positive real points in  $S$}. 

Now suppose $\vert \xi \vert \in (0,\infty) \setminus {\cal E}_0$  
and let $\psi(\cdot,\xi) \in L^2$
be the solution to (\ref{LS_mod_eq}). Since $V^{1/2}(\cdot) \psi(\cdot,\xi)
\in L^1$, the integral 
$$
\int_{-\infty}^\infty e^{i \vert \xi \vert  
\vert x-y \vert} V^{1/2}(y) \psi(y,\xi) \, dy
$$ 
converges for all $x$, and defines a bounded (locally) absolute continuous function of $x$. 
Note that
the function $e(x,\xi) \in AC_{loc} $ defined in the second statement of the lemma satisfies
\begin{align}
V(x) e(x,\xi) =& V^{1/2}(x) \psi(x,\xi) \in L^1  \label{eq:Ve} 
\end{align} 
and that (\ref{LS_mod_eq}) implies that $\mathrm{supp} \ \psi(\cdot, \xi) \subset \mathrm{supp}\  V(\cdot) $ 
if   $| \xi | \in \R\setminus {\cal E}_0$.  
It follows immediately that $e(x,\xi)$ is a solution to 
equation (\ref{LS_eq}). Finally, according to the remark prior to  (\ref{LS_mod_eq}) 
(or by a direct calculation), we see that the solution $e(x,\xi)$ is in fact
differentiable a.e in the $x$ variable.  \end{proof}  

The existence of $(I {\color{red}+}T(w))^{-1}$ on $\C_+ \setminus S$ depends on the following variant of analytic Fredholm theorem.

\begin{lemma}\label{lem:Fredholm} If $T(z)$ is a family of compact operators that is
analytic in $U={\Im z>0}$ and continuous on 
$\partial U\setminus \{0\} $, then either $(I+T(w))^{-1}$ does not exist
for any $z\in {U}$ or  $(I+T(w))^{-1}$ exist for all $z\in \overline{U}$ 
except
a closed subset (counting the origin) in $U\cup\partial U$ of measure zero.
\end{lemma}


\begin{remark}\label{re:mero-fredholm}  This is a slightly improved version of the Fredholm theorem. 
The analytic Fredholm of this type  can be found in \cite[XI.6]{RS}. 
For such results on a general complex domain,  see   \cite[Theorem 1.9, VII.1]{Kato66}. 
For a meromorphic version see Theorem XIII.13 in \cite{RS}. 
\end{remark}

\vspace{.2in}
\noindent 
{\bf Notation:} Let ${\cal E} = \{ \xi \in \R: \vert \xi \vert \in 
{\cal E}_0 \} \cup \{ 0 \}$.  
Henceforth, when $ \xi \in  \R 
\setminus {\cal E}$, $e(x,\xi)$ 
denotes the solution to equation 
(\ref{LS_eq}) constructed in Lemma \ref{l:e(xxi):E0}. 
Note that $\cE_0\cup \{0\}$ is a bounded closed set of measure zero. 

It is easy to show that any  embedded eigenvalues, if any, would belong to the exceptional set $\cE_0$  \cite[XI, Problem 62]{RS}. 
\begin{lemma}\label{lem:+ve-eigenval} Let $E>0$ with $\sqrt{E}\notin \cal{E}_0$. 
Then $E$ is not an eigenvalue of $H$.    
\end{lemma}

\vspace{.20in}

\section{The resolvent for $H_V$}\label{s:resolvent:Hv}

The next few lemmas concern the resolvent of $H=H_V$ when $V\in L^1\cap L^2$ is real-valued. 
From the beginning of section \ref{s:L2} we know that  $\sigma(H) =\sigma_{ac}\cup \sigma_{pp}$, 
where $\sigma_{ac}(H)=[0,\iy)$, $\sigma_{sc}(H)=\emptyset$ and
$\sigma_{pp}(H)=\{\lam_1, \lam_{2},\dots\}\subset  (-\infty, 0)$. 

\vspace{.20in}

If $z \in \C \setminus [0,\infty)$, we let 
$R(z) = (H - z)^{-1}$ and $R_0(z) = (H_0 - z)^{-1}$. The unperturbed resolvent
$R_0(z)$ has an integral kernel 
 $$
G_0(x,y;z) = -(2i z^{1/2})^{-1} e^{i z^{1/2} \vert x-y \vert}
$$
where the value of the square root $z^{1/2}$ is taken to have positive 
imaginary part. 

\begin{lemma}\label{L:inverse} Let $V\in L^1$. 
Let $z \in \C \setminus \sigma(H)$.
Then the operator $I + \vert V \vert^{1/2} R_0(z) V^{1/2} \in {\cal B}
(L^2)$ has a bounded inverse. 
\end{lemma} 

\noindent
\begin{proof} 
Let $T_z$ denote $I + \vert V \vert^{1/2} R_0(z) V^{1/2}$. 
We  prove that $T_z$ is a bijection on $L^2$. 
We divide the proof of the lemma in two steps.
\begin{enumerate}
\item[(a)]\label{ia}  First we assume $V\in L^1\cap L^\iy$. 
Suppose $f \in L^2$ is such that $T_{z}f = 0$.  
 This means that $\vert V \vert^{1/2} R_0(z) V^{1/2} f = -f$.
 Let $g = R_0(z) V^{1/2} f$ . Since $V^{1/2}f\in L^2$, we see  $g\in {\cal D}(H_0)$. 
  It is easy to verify that  $Vg = -(H_0 - z)g$. Hence 
\begin{equation}\label{g:H:z}
 Hg=zg \quad in \ L^2\,.
 \end{equation} 
 Since  $z$ is not in the spectrum of $H$, it must be that $g = 0$. 
 Since $f = -\vert V \vert^{1/2}g$, it follows that $f = 0$. This shows that $T_{z}$ is injective. 
  A similar calculation shows that $T_{z}$ is surjective by noting that 
  $f= (I- |V|^{1/2} R(z) V^{1/2}) h $ is a solution to $ T_{z}f=h $ 
  in view of the second resolvent identity (\ref{resolvent_eq}).
  Thus, it follows from the inverse mapping theorem that $T_{z}$ has a bounded inverse.   

\item[(b)]\label{i:(b)} Now  we relax the condition in (a) to assume $V\in L^1$ only.   
We shall prove the results in part (a)
using a limiting argument  with $C^\infty_0\ni V_n {\rightarrow} V$ a.e. and in  $L^1$.
Let $f\in \mathrm{Ker} \ T_z$.  Let $g=R_0(z)V^{1/2}f$.  
First show $g\in D(H)\subset Q(H)=Q(H_0)$. 
Let $C^\infty_0\supset \{V_n\}\rightarrow V\;\text{ in}\; L^1\; \text{and}\; a.e.$
with $|V_n(x)|\le |V(x)|$. 
Then $g_n=R_0(z)V_n^{1/2}f \in D(H_0)$. 
Note $R_0(z)V_n^{1/2}f \rightarrow R_0(z){V^{1/2}}f=g\in L^2$ 
(since $V^{1/2}f\in L^1 $ and $R_0(z): L^1\rightarrow L^q$, $\forall q\ge 1$). 
 The convergence is in operator norm for  $R_0(z)V_n^{1/2}$ hence in $L^2$ when it is applied to $f$.  
This can be easily observed from the kernel 
of $R_0(z)V_n^{1/2}-R_0(z)V^{1/2}$ that is equal to 
\begin{equation*}
-\frac{1}{2iz} e^{iz^{1/2}|x-y|} (V_n^{1/2}-V^{1/2})\,,
\end{equation*}
which actually converges to zero in the Hilbert-Schmidt norm.  We have for all $h\in C^\infty_0$,  
\begin{align*}
&V_n^{1/2}f =(H_0-z)g_n \ {\Rightarrow} \ \langle V_n^{1/2}f, h\rangle
=\langle (H_0-z)g_n, h\rangle\\
=&\langle g_n, (H_0-{z}) h\rangle
\rightarrow \langle g, (H_0-{z}) h\rangle\,.
\end{align*}
On the left hand side, $\int V_n^{1/2}f h \rightarrow \int V^{1/2}f h dx = \la V^{1/2}f, h\ra$ 
by Lebesgue's dominated convergence theorem, noting that $V^{1/2}f\in L^1$.  
We have shown that 
$$
L^1\ni V^{1/2}f= (H_0-z)g \qquad \text{in the sense of distributions}.
$$
 Now $T_z f=0$ implies that $ Vg=-V^{1/2}f$, i.e., 
\[  \la Vg, h\ra= -\la (H_0-z)g, h\ra. \]
Therefore,  $Vg= -(H_0-z)g $ in distributions.
This  implies $g\in D(H)$ by regularity theorem for elliptic equations since $Vg=-V^{1/2}f\in L^2$. 
The remaining part of the proof goes the same as in part (a).
\end{enumerate} 
\end{proof}






\begin{lemma}\label{L:R(z)}  Let $V\in L^1\cap L^2$. 
Let $z \in \C \setminus \sigma(H)$. 
Then there exists a measurable function $G(x,y;z)$ on $\R \times \R$ such 
that 
$$ 
(R(z)f)(x) = \int_{-\infty}^\infty G(x,y;z) f(y) \, dy
$$
for all $f \in L^2$. Moreover: \n\n
\noindent 
(i) $G(x,\cdot \, ;z) \in L^1 \cap L^2$ for almost every 
$x \in \R$. \n\n 
\noindent
(ii) $G(x,y;z) = G(y,x;z)$ and $\overline{G(x,y;z)} = 
G(x,y;\overline{z})$ for almost every $(x,y) \in \R^2$.  \n\n 
\noindent
(iii) $G$ satisfies the equation 
$$
G(x,y;z) = G_0(x,y;z) - \int_{-\infty}^\infty G_0(x,w;z) V(w) 
G(w,y;z) \, dw
$$
for almost every $(x,y) \in \R^2$. 
\end{lemma} 

\begin{proof}  Since $D(V)\supset D(H_0)$ using the resolvent equation 
\begin{equation} 
R(z) = R_0(z) - R_0(z) V R(z) 
\label{resolvent_eq}
\end{equation} 
and Lemma \ref{L:inverse}, it is easy to check the identity
$$
R(z) - R_0(z) = -R_0(z) V^{1/2}(I + \vert V \vert^{1/2} R_0(z) V^{1/2})^{-1}  
\vert V \vert^{1/2} R_0(z) 
$$
Since $V$ is integrable, the operator $\vert V \vert^{1/2} R_0(z)$ is 
Hilbert-Schmidt. Therefore $R(z) - R_0(z)$ is Hilbert-Schmidt, and has
an integral kernel $A(x,y;z)$, measurable as a function of 
$(x,y) \in \R^2$, such that  
$$
\int_{\R^2} \vert A(x,y;z) \vert^2 \, dx \times dy < \infty. 
$$ 
In particular 
$A(x,\cdot \, ;z) \in L^2$ for almost every $x$. 
Since $G_0(x,\cdot \, ;z) \in 
L^2$ for all $x$, $R(z)$ has an integral kernel 
$$
G(x,y;z) = A(x,y;z) + G_0(x,y;z), 
$$
measurable as a function of $(x,y) \in \R^2$, such that
$G(x,\cdot \, ;z) \in L^2$ for almost every $x$. 

Now according to the 
resolvent equation (\ref{resolvent_eq}), if $f,g \in L^2$, then 
\begin{eqnarray}
\lefteqn{ 
\int_{-\infty}^\infty \int_{-\infty}^\infty f(x) G(x,y;z) g(y) \, dy dx 
}  \nonumber \\ 
& & = \int_{-\infty}^\infty \int_{-\infty}^\infty f(x) G_0(x,y;z) g(y)
\, dy dx \nonumber \\ 
& & \ \ \ \ \ \ \ - \int_{-\infty}^\infty
\int_{-\infty}^\infty \int_{-\infty}^\infty f(x) G_0(x,w;z) V(w) 
G(w,y;z) g(y) \, dy dw dx. \label{G_eq}
\end{eqnarray}
These are {\em iterated} integrals with the indicated order of integration. 
However, let $c,\epsilon > 0$ be chosen such that $\vert G_0(x,y;z) \vert 
\leq c e^{-\epsilon \vert x-y \vert}$. Define $a \in L^1$ by $a(x) = 
e^{-\epsilon \vert x \vert}$. Then  
$$
\int_{-\infty}^\infty \int_{-\infty}^\infty \vert f(x) G_0(x,y;z) g(y)
\vert \, dy dx \leq c \int_{-\infty}^\infty \vert f(x)
\vert (a * \vert g \vert)(x) \, dx \leq c \Vert f \Vert_{L^2} \Vert g
\Vert_{L^2}.
$$
This shows that 
$$
\int_{-\infty}^\infty \int_{-\infty}^\infty f(x) G_0(x,y;z) g(y) \, dy dx 
= \int_{\R^2} G_0(x,y;z) \, f(x)g(y) \, dx \times dy.
$$
The bound 
$$
\int_{-\infty}^\infty \int_{-\infty}^\infty \vert f(x) A(x,y;z) g(y) \vert
\, dy dx \leq \Vert g \Vert_{L^2} \Vert f \Vert_{L^2}
( \int_{-\infty}^\infty \int_{-\infty}^\infty
\vert A(x,y;z) \vert^2 \, dy dx)^{1/2}
$$
shows that 
$$
\int_{-\infty}^\infty \int_{-\infty}^\infty f(x) G(x,y;z) g(y) \, dy dx 
= \int_{\R^2} G(x,y;z) \, f(x)g(y) \, dx \times dy.
$$
Finally, by similar arguments the triple integral 
in (\ref{G_eq}) can be written in the form 
$$
\int_{\R^2} \left( \int_{-\infty}^\infty G_0(x,w;z) V(w) G(w,y;z) 
\, dw \right) \, 
f(x)g(y) \, dx \times dy.
$$
The function $$
f(x)g(y) \, \int_{-\infty}^\infty G_0(x,w;z) V(w) G(w,y;z) \, dw
$$
is measurable as a function of $(x,y) \in \R^2$ whenever $f,g \in L^2$. 
It follows that 
$$
\int_{-\infty}^\infty G_0(x,w;z) V(w) G(w,y;z) \, dw
$$ 
itself is a measurable as a function of $(x,y) \in \R^2$. Now define
$$
B(x,y;z) = G(x,y;z) - G_0(x,y;z) + \int_{-\infty}^\infty G_0(x,w;z)
V(w) G(w,y;z) \, dw. 
$$
So far we know that if $f,g \in L^2$, then 
\begin{equation} 
\int_{\R^2} B(x,y;z) \, f(x)g(y) \, dx \times dy = 0. \label{B_eq}
\end{equation}  
For $M > 0$ define $B_M(x,y;z) = \chi_{[-M,M]^2}(x,y) B(x,y;z)$. It is easy
to check that $B_M \in L^2(\R^2)$ for each $M$. By a standard argument 
(\ref{B_eq}) implies that 
$$
\int_{[-M,M]^2} B_M(x,y;z) F(x,y) \, dx \times dy = 0
$$
for every $F \in L^2([-M,M]^2)$. Therefore $B_M(x,y;z) = 0$ for almost every 
$(x,y) \in \R^2$. Since this holds for all $M$, $B(x,y;z) = 0$ for 
almost every $(x,y) \in \R^2$. This proves $(iii)$.

To prove ${\emph (ii)}$, let $f,g \in L^2$ and write
\begin{eqnarray*}
\int_{-\infty}^\infty f(x) \, (R(z)g)(x) \, dx & = & 
\int_{-\infty}^\infty ((H-z)R(z)f)(x) \, (R(z)g)(x) \, dx \\ 
& = & \int_{-\infty}^\infty (R(z)f)(x) \, ((H-z)R(z)g)(x) \, dx \\
& = & \int_{-\infty}^\infty (R(z)f)(x) \, g(x) \, dx.
\end{eqnarray*}
This implies that $G(x,y;z) = G(y,x;z)$ for almost every 
$(x,y) \in \R^2$. 
Also, $R(z)^* = R(\overline{z})$, 
so $G(x,y;z) = \overline{G(y,x;\overline{z})}$ 
for almost every $(x,y) \in \R^2$. This proves ${\emph (ii)}$.

It remains to prove that $G(x,\cdot \, ;z) \in L^1$ for almost every $x$. 
Using the integral equation ${\emph (iii)}$ and the symmetry of $G(x,y;z)$
we have, for almost every $x$ 
\begin{equation}\label{G_est} 
\begin{aligned}
&\int_{-\infty}^\infty \vert G(x,y;z) \vert \, dy = \int_{-\infty}^\infty \vert G(y,x;z) \vert \, dy  \\ 
 \leq& \int_{-\infty}^\infty \vert G_0(y,x;z) \vert \, dy + 
\int_{-\infty}^\infty \int_{-\infty}^\infty \vert G_0(y,w;z) V(w) 
G(w,x;z) \vert \, dw dy.  
\end{aligned}\end{equation} 
Again let $c,\epsilon > 0$ be chosen such that $\vert G_0(y,x;z) \vert
\leq c e^{-\epsilon \vert x-y \vert}$. Define functions $a(w) = e^{-\epsilon
\vert w \vert}$ and $b(w) = \vert V(w) G(w,x;z) \vert$. Note that 
$a \in L^1$, and $b \in L^1$ for almost every $x$. The first integral in 
(\ref{G_est}) is bounded by $c \Vert a \Vert_{L^1}$ for all $x$. The second
integral in (\ref{G_est}) is bounded by $c \Vert a \Vert_{L^1} 
\Vert b \Vert_{L^1}$, which is finite for almost every $x$. \end{proof}

\begin{remark}\label{rem:L1V:HS}
 As soon as  $V\in L^1$,  from the proof of 
Lemma \ref{L:R(z)} we note that for $z\in \C\setminus \sigma(H)$,
$$ R(z)-R_0(z) $$ 
is Hilbert-Schmidt,  hence compact, which  provides the essential spectral information for $H$ as well. 
\end{remark}

\begin{lemma}\label{lem:phi-G} Let $V\in L^1\cap L^2$.
If $z \in \C \setminus \sigma(H)$, 
then the function 
$$
g(x) = \vert V\vert^\frac{1}{2}(x) \int_{-\infty}^\infty \vert G(x,y;z) \vert \, dy
$$
belongs to $L^2(\R)$. 
\end{lemma}

\begin{proof}  

Let $\epsilon = {\cal I}m\ z^{1/2} >0 $. Using the integral equation for $G$
and the symmetry of $G$, we bound 
$$
\Vert g \Vert_{L^2}^2 = \int_{-\infty}^\infty \int_{-\infty}^\infty
\int_{-\infty}^\infty \vert V(x) \vert \, \vert G(y_1,x;z) \vert
\,\vert G(y_2,x;z) \vert \, dy_1 dy_2 dx
$$
by a finite combination of terms of the form
$$
\text{I}= \int_{-\infty}^\infty \int_{-\infty}^\infty \int_{-\infty}^\infty
\vert V(x) \vert e^{-\epsilon \vert y_1 - x \vert} 
e^{-\epsilon \vert y_2 - x \vert} \, dy_1 dy_2 dx,
$$
$$
\textrm{II}= \int_{-\infty}^\infty \int_{-\infty}^\infty \int_{-\infty}^\infty
\int_{-\infty}^\infty \vert V(x)  
\vert e^{-\epsilon \vert y_1 - t \vert}
\vert V(t) \vert \, \vert G(t,x;z) \vert 
e^{-\epsilon \vert y_2 - x \vert} \, dt dy_1 dy_2 dx,
$$
and 
\begin{eqnarray*}
\lefteqn{ 
\mathrm{III}= \int_{-\infty}^\infty \int_{-\infty}^\infty \int_{-\infty}^\infty
\int_{-\infty}^\infty \int_{-\infty}^\infty \vert V(x) \vert 
e^{-\epsilon \vert y_1 - t_1 \vert} \vert V(t_1) \vert \, \vert
G(t_1,x;z) \vert  } \\ 
& & \ \ \ \ \ \ \ \ \ \ \ \ \ \ \ \ \ \ \ \ \ \ \cdot
e^{-\epsilon \vert y_2 - t_2 \vert} \vert V(t_2) \vert
\, \vert G(t_2,x;z) \vert \, dt_1 dt_2 dy_1 dy_2 dx.
\end{eqnarray*} 
Clearly term I is finite. 
Term II is bounded by
\begin{eqnarray*}
\lefteqn{
c \, \int_{-\infty}^\infty \int_{-\infty}^\infty \vert V(x) \vert
\, \vert V(t) \vert \, \vert G(t,x;z) \vert \, dt dx } \\
& & \ \ \ \ \ \ \ \ \ \ \leq  
c \, \int_{-\infty}^\infty \int_{-\infty}^\infty \vert V(x) \vert
\, \vert V(t) \vert \, e^{-\epsilon \vert t-x \vert} \, dt dx \\ 
& & \ \ \ \ \ \ \ \ \ \ \ \ \ 
+ c \, \int_{-\infty}^\infty \int_{-\infty}^\infty 
\vert V(x) \vert \, \vert V(t) \vert \, \vert A(t,x;z) \vert
\, dt dx.
\end{eqnarray*}
The first integral on the right side is finite, 
and the second is bounded by 
$$
\Vert V \Vert_{L^2} \Vert V \Vert_{L^2} 
\left( \int_{-\infty}^\infty
\int_{-\infty}^\infty \vert A(t,x;z) \vert^2 \, dtdx \right)^{1/2}
< \infty.
$$
Term III is bounded by a combination of terms of the form 
\begin{eqnarray*}
\lefteqn{
\text{IV}= \int_{-\infty}^\infty \cdots \int_{-\infty}^\infty \vert V(x) \vert
e^{-\epsilon \vert y_1 - t_1 \vert} \vert V(t_1) \vert 
e^{-\epsilon \vert t_1 - x \vert} } \\
& & \ \ \ \ \ \ \ \ \ \ \ \ \ \ \ \cdot 
e^{-\epsilon \vert y_2 - t_2 \vert} \vert V(t_2) \vert
e^{-\epsilon \vert t_2 - x \vert}  \, dt_1 dt_2 dy_1 dy_2 dx,
\end{eqnarray*} 
\begin{eqnarray*}
\lefteqn{
\textrm{V}= \int_{-\infty}^\infty \cdots \int_{-\infty}^\infty \vert V(x) \vert
e^{-\epsilon \vert y_1 - t_1 \vert} \vert V(t_1) \vert
\, \vert A(t_1,x;z) \vert } \\
& & \ \ \ \ \ \ \ \ \ \ \ \ \ \ \ \cdot
e^{-\epsilon \vert y_2 - t_2 \vert} \vert V(t_2) \vert
e^{-\epsilon \vert t_2 - x \vert}  \, dt_1 dt_2 dy_1 dy_2 dx,
\end{eqnarray*}
and
\begin{eqnarray*}
\lefteqn{
\textrm{VI}= \int_{-\infty}^\infty \cdots \int_{-\infty}^\infty \vert V(x) \vert
e^{-\epsilon \vert y_1 - t_1 \vert} \vert V(t_1) \vert
\, \vert A(t_1,x;z) \vert } \\
& & \ \ \ \ \ \ \ \ \ \ \ \ \ \ \ \cdot
e^{-\epsilon \vert y_2 - t_2 \vert} \vert V(t_2) \vert \, 
\vert A(t_2,x;z) \vert  \, dt_1 dt_2 dy_1 dy_2 dx.
\end{eqnarray*}
 We can immediately integrate
term IV, with a finite result. Term V is bounded by 
$$
c \, \int_{-\infty}^\infty \int_{-\infty}^\infty 
\vert V(x) \vert \, \vert V(t_1) \vert \, 
\vert A(t_1,x;z) \vert \, dt_1 dx < \infty.
$$
Finally, term VI is bounded by 
\begin{eqnarray*}
\lefteqn{
c \, \int_{-\infty}^\infty \int_{-\infty}^\infty \int_{-\infty}^\infty
\vert V(x) V(t_1) V(t_2) \vert \, 
\vert A(t_1,x;z) \vert \, \vert A(t_2,x;z) \vert
\, dt_1 dt_2 dx } \\ 
& & \ \ \ \ \ \ \ \ \ \leq 
c \left( \int_{-\infty}^\infty \int_{-\infty}^\infty \int_{-\infty}^\infty
\vert V(x) V(t_1) V(t_2) \vert \,
\vert A(t_1,x;z) \vert^2  \, dt_1 dt_2 dx \right)^{1/2} \\
& & \ \ \ \ \ \ \ \ \ \ \ \ \ \cdot
\left( \int_{-\infty}^\infty \int_{-\infty}^\infty \int_{-\infty}^\infty
\vert V(x) V(t_1) V(t_2) \vert \,
\vert A(t_2,x;z) \vert^2  \, dt_1 dt_2 dx \right)^{1/2},
\end{eqnarray*}
which is finite. 
\end{proof}

\begin{lemma}\label{lem:F-Green} {Let $V\in L^1\cap L^2$}.
Let $\xi \in \R$, $\kappa \in \C$, ${\Im}\ \kappa > 0$.  
 $\kappa^2\notin \sigma_{pp}(H)$.
Define 
\begin{equation}\label{h(xi:ka):G}
h(x,\xi;\kappa) = (\xi^2 - \kappa^2) \int_{-\infty}^\infty G(x,y;\kappa^2) e^{iy\xi} \, dy
\end{equation}
and $$
p(x,\xi;\kappa) = \vert V(x) \vert^{1/2} h(x,\xi;\kappa).
$$
Then the function $h(x,\xi;\kappa)$ satisfies 
$$
h(x,\xi;\kappa) = e^{ix\xi} + (2i\kappa)^{-1} \int_{-\infty}^\infty
e^{i\kappa \vert x-y \vert} V^{1/2}(y)
p(y,\xi;\kappa) \, dy
$$
for a.e. $x$. The function $p(\cdot \, ,\xi;\kappa)$ belongs to 
$L^2$ and satisfies 
$$
p(x,\xi;\kappa) = \vert V(x) \vert^{1/2} e^{ix\xi} + (2i\kappa)^{-1} 
\int_{-\infty}^\infty \vert V(x) \vert^{1/2} e^{i\kappa \vert x-y \vert}
V^{1/2}(y) p(y,\xi;\kappa) \, dy.
$$
for a.e. $x$. Furthermore, $p(\cdot \, ,\xi;\kappa)$ 
is the unique $L^2$ solution to this equation, 
with $\Vert p(\cdot \, ,\xi;\kappa)\Vert_2 \le c_\kappa 
$.
\end{lemma} 

\begin{proof} Note a priorie that  $p(\cdot \, ,\xi;\kappa)$ 
is in $L^2$ by Lemma \ref{lem:phi-G}.  
According to Lemma \ref{L:R(z)} (\emph{iii}), for almost every $x$
$$
h(x,\xi;\kappa) = e^{ix\xi} + 
(\xi^2 - \kappa^2) (2i\kappa)^{-1} \int_{-\infty}^\infty
\int_{-\infty}^\infty e^{i\kappa \vert x-w \vert} V(w) G(w,y;\kappa^2) 
e^{iy\xi} \, dw dy.
$$
An application of Lemma \ref{lem:phi-G} shows that the order of the 
two integrations can be exchanged, resulting in
$$  
h(x,\xi;\kappa) = e^{ix\xi} + (2i\kappa)^{-1} \int_{-\infty}^\infty
e^{i\kappa \vert x-w \vert} V(w) h(w,\xi;\kappa) \, dw 
$$  
for a.e. $x$. 
The integral equations for $h$ and $p$ stated in the 
lemma follow immediately from this equation. 

Finally, {since}  $p(\cdot \, ,\xi;\kappa)$ 
is in $L^2$ by Lemma \ref{lem:phi-G}, the integral
equation for $p$ has the form
$$
(I + \vert V \vert^{1/2} R_0(\kappa^2) V^{1/2})a = b, 
$$
where $b \in L^2$ denotes $\vert V(x) \vert^{1/2} e^{ix\xi}$ 
and $a$ represents the 
solution $p(\cdot \, ,\xi;\kappa)$. According to Lemma \ref{L:inverse} this equation has a unique solution.  
\end{proof}


\begin{lemma}\label{lem:Phi}  {Let $V\in L^1\cap L^2$}.
Let $f \in C^\infty_0$. Let $\xi \in \R$, $\kappa \in \C$, 
${\Im}\ \kappa > 0$ with  $\kappa^2 \notin \sigma_{pp}(H)$.
Define 
$$
\Phi(\xi;\kappa) = (2\pi)^{-1/2} \int_{-\infty}^\infty f(x)
\overline{h(x,\xi;\kappa)} \, dx 
$$
and for a.e. $\xi$  
\begin{equation}\label{e:fsh}
{f}^\sharp (\xi) = (2\pi)^{-1/2} \int_{-\infty}^\infty f(x) \overline{e(x,\xi)} \, dx.
\end{equation}
Suppose $0 < \alpha < \beta$ and $[\sqrt{\alpha},\sqrt{\beta}] \cap {\cal E}_0 = \emptyset$.
Then $\Phi(\xi;\kappa)$
has an extension to real $\kappa$ for $\sqrt{\alpha} \le \kappa \leq \sqrt{\beta}$. If $\alpha \leq  \xi^2 \leq \beta$,
then $f^\sharp (\xi) = \Phi(\xi;\vert \xi \vert)$. Moreover, 
the extended version of $\Phi$ is bounded and 
uniformly continuous in $\xi \in \R$, 
$\sqrt{\alpha} \leq {\Re}\ \kappa \leq \sqrt{\beta}$, $ 
{\cal I}m\ \kappa \geq 0$. 
\end{lemma}

\begin{proof} Lemma \ref{lem:F-Green} implies that the integral defining $\Phi(\xi;\kappa)$ converges
absolutely, and Lemma \ref{l:e(xxi):E0} implies that the integral defining $f^\sharp(\xi)$
converges absolutely. Consider the operator $(I + T(\kappa)) \in {\cal B}(L^2)$
where $T(\kappa)$ has integral kernel 
$$
T_\kappa(x,y) = -(2i\kappa)^{-1} \vert V(x) \vert^{1/2} e^{i\kappa\vert 
x-y \vert} V^{1/2}(y).
$$
If ${\cal I}m\ \kappa > 0$ and
 $\kappa^2 \notin \sigma_{pp}(H)$, then $T(\kappa) = \vert V \vert^{1/2} 
R_0(\kappa^2) V^{1/2}$, in which case Lemma \ref{L:inverse} implies that 
$(I + T(\kappa))^{-1} \in {\cal B}(L^2)$ exists. On the other hand, if 
$\kappa$ is real and positive, and $\kappa \notin {\cal E}_0$, 
then the proof of Lemma \ref{l:e(xxi):E0}
shows that $(I + T(\kappa))^{-1} \in {\cal B}(L^2)$ also exists. The family
of Hilbert-Schmidt operators $T(\kappa)$ is ${\cal B}(L^2)$ norm continuous
as a function of $\kappa \in \bar{U} \setminus \{ 0 \}$. 
$U=\{z\in \C: {\Im} z>0\}$.
Therefore, the family of operators
$(I + T(\kappa))^{-1}$ is uniformly continuous 
(in ${\cal B}(L^2)$ norm) on the strip 
$K = \{ \kappa \in \C : \sqrt{\alpha} \leq {\Re}\ \kappa \leq \sqrt{\beta},  {\Im}\ \kappa  \geq 0 \}$. 

Define $F$: $\R \rightarrow L^2(\R)$ by the rule  
$F(\xi) = \vert V(x) \vert^{1/2} e^{ix\xi}$. For a.e. $x$, $F$ is a bounded, uniformly
continuous function of $\xi$. It follows that the $L^2(\R)$-valued function 
$(I + T(\kappa))^{-1} F(\xi)$ is 
bounded and uniformly continuous as a function of 
$(\xi,\kappa) \in \R \times K$.  
If ${\cal I}m\ \kappa > 0$, then $(I + T(\kappa))^{-1} F(\xi)$
is just $p(\cdot \, ,\xi;\kappa)$. So, we know that the family of 
$L^2$ functions $p(\cdot \, ,\xi;\kappa)$ has a 
bounded, uniformly continuous (in $L^2$ norm) extension
to the parameter set $(\xi,\kappa) \in \R \times K$.  

Now for ${\Im}\ \kappa > 0$ and $\kappa^2\notin \sigma_{pp}$ define 
$$
\tilde{h}(x,\xi;\kappa) = e^{ix\xi} + (2i\kappa)^{-1}
\int_{-\infty}^\infty e^{i\kappa \vert x-y \vert} V^{1/2}(y) 
p(y,\xi;\kappa) \, dy.
$$
Recall that according to Lemma \ref{lem:F-Green}, $\tilde{h}(x,\xi;\kappa)
= h(x,\xi;\kappa)$ for almost every $x$. 
Let $C_b:=C^0_b(\R)$ denote the Banach space of bounded continuous functions on 
$\R$ with sup norm. It is not difficult to check that the family of 
functions $\vert f(\cdot) \vert^{1/2} \tilde{h}(\cdot \, ,\xi;\kappa)
\in C_b$ has a bounded, uniformly continuous (in $C_b$ norm)
extension to $(\xi,\kappa) \in \R \times K$. 
It follows that $\Phi(\xi;\kappa)$ itself has a bounded, uniformly
continuous extension to $(\xi,\kappa) \in \R \times K$.  

Finally, if $\alpha \leq  \xi^2 \leq \beta$, 
then inspection of the proof of 
Lemma \ref{l:e(xxi):E0} shows that the extended version of $p$ satisfies 
$p(x,\xi;\vert \xi \vert) = 
\psi(x,\xi)$, (this can also been seen by the uniqueness of solution to (\ref{LS_mod_eq})).  Comparing the definition of $\tilde{h}$ and the integral 
equation for $e(x,\xi)$ in Lemma \ref{l:e(xxi):E0}, we find that the extended version of 
$\Phi$ has the property ${f}^\sharp(\xi) = \Phi(\xi;\vert \xi \vert )$. 
\end{proof}

\noindent   
\begin{remark}\label{e(xxi):alpha-beta} 
 Suppose $0 < \alpha < \beta$ and $[\sqrt{\alpha},\sqrt{\beta}] \cap {\cal E}_0 = \emptyset$.  
The proof of  Lemma \ref{lem:Phi} shows that the family of $L^2$ functions $\psi(\cdot,\xi)$ is bounded
and uniformly continuous (in $L^2$ norm) for 
$\alpha \leq  \xi^2 \leq \beta$. 
It follows from the integral equation for $e(x,\xi)$ in 
Lemma  \ref{l:e(xxi):E0} that there
exists some $c_{\alpha,\beta} < \infty$ such that 
$\vert e(x,\xi) \vert \leq c_{\alpha,\beta}$ for all $x \in \R$,  
$\alpha \leq  \xi^2  \leq \beta$. 
With a little work it also follows from the integral equation
that $e(x,\xi)$ is jointly continuous in $x,\xi$ for $x \in \R$,  $\alpha \leq \xi^2 \leq \beta$. 
\end{remark}

\section{Proof of Theorem \ref{th:Plancherel}}\label{s:main:proof} 

The main procedure 
 to prove the theorem is to  show  Lemma \ref{lem:Pf},  
  a Plancherel type formula for $H_V$. 
The proof of the lemma are based on results in \S \ref{s:L2} to \S \ref{s:resolvent:Hv},
where we have used the resolvents to construct the generalized eigenfunctions $e(x,\xi)$ and the spectral measure for $H_V$ in Lemma \ref{l:e(xxi):E0} to
Lemma \ref{lem:Phi}.
Parts (a) and (b) of Theorem \ref{th:Plancherel}  establish the completeness of $e(\cdot,\xi)$ or equivalently, the inversion formula for $\cF$ on $\sH_{ac}$
\begin{align}\label{eF:Eac}    
    \cF^*\cF=P_{ac} \,. 
 \end{align}

 
For $A \subset \R$ a Borel set and $\chi_A$ the characteristic function of $A$, 
let $P_A$  denote the operator $\chi_A(H)$ defined by the functional calculus: $P_A=\int \chi_A(\lam) dE_\lam $.   

\begin{lemma}\label{lem:Pf}  
Let $f \in C^\infty_0$. Suppose $0 < \alpha < \beta$ are such that 
$[ \sqrt{\alpha},\sqrt{\beta} \, ] \cap {\cal E}_0 = \emptyset$. Then  
\begin{align}\label{L2:Pf:Ff}
\Vert P_{[\alpha,\beta)} f \Vert_{L^2}^2 = 
\int_{\alpha \leq \xi^2 \leq \beta} \vert {f}^\sharp (\xi) \vert^2 \, d\xi.   
\end{align}
\end{lemma} 

\begin{remark}\label{rem:f>L2} 
This identity holds for all $f$ in $L^2$ by a density argument in light of Lemma \ref{lem:Phi}. 
\end{remark}


\begin{remark}\label{rem:ac-Elam} Let  $[ \sqrt{\alpha},\sqrt{\beta} \, ] \cap {\cal E}_0 = \emptyset$.
Through polarization, equation (\ref{L2:Pf:Ff}) implies that  for all $f$ and $g$ in $L^1\cap L^2$, 
\begin{align*}
&( P_{[\alpha,\beta)} f, g ) = 
\int_{\alpha \leq \xi^2 \leq \beta}  {f}^\sharp (\xi) \overline{g^\sharp(\xi)} d\xi \\ 
=& \int_{\R^2} f(y) \overline{g(x)} k_{\al,\beta}(x,y)  dxdy \,,
\end{align*}
where  $k_{\al,\beta}(x,y)  =\int_{\alpha \leq \xi^2 \leq \beta} e(x,\xi)\overline{e (y,\xi)} d\xi\,$. 
Thus, we derive from the spectral resolution $H P_{ac}=\int_{0}^\iy \lam dE_\lam=\int_{\R_+\setminus \cE_0^2} \lam dE_\lam$
\begin{align}\label{dE:density-meausre}
& 
\frac{d }{d\lam} ( E_\lam f,g) = \int_{\R^2}  p(x,y,\lam) f(y)\overline{g(x)} dxdy \,,
\end{align}
where $p(x,y,\lam)=\frac1{2\sqrt{\lam}} \left(e(x,\sqrt{\lam})\overline{e (y,\sqrt{\lam})} +e(x,-\sqrt{\lam})\overline{e (y,-\sqrt{\lam}) }\right)$. 
Since $e(x,\xi)$ is continuous in $(x,\xi)\in \R\times [\al,\beta]$ 
according to Remark \ref{e(xxi):alpha-beta}, it follows that the measure  $ \mu_{f,g}(\lam):=( E_\lam f,g)$ is 
absolutely continuous for all $\lam\in \R_+\setminus \cE_0^2$. 
\end{remark} 

\vspace{.080in}

\begin{proof} Let $\xi \in \mathbb{R}$,  $\kappa \in \C$, ${\cal I}m\ \kappa > 0$, 
$\kappa^2\notin \sigma_{pp}(H)$  and 
write $\kappa^2 = \mu + i \epsilon$, with $\mu,\epsilon \in \R$. 
According to Lemma \ref{L:R(z)}, $G(x,\cdot \, ;\kappa^2) \in L^1 \cap L^2$ for almost
every $x$. Applying  Plancherel's theorem to the functions in (\ref{h(xi:ka):G}) we have
\begin{equation}
(\kappa^2 - \overline{\kappa}^2) \int_{-\infty}^\infty \overline{G(x,w;
\overline{\kappa}^2)} G(y,w;\overline{\kappa}^2) \, dw  =  
\frac{1}{\pi} \int_{-\infty}^\infty \frac{i\epsilon}{(\xi^2 - \mu)^2
+ \epsilon^2} h(x,\xi;\kappa) \overline{h(y,\xi;\kappa)} \, d\xi
\label{G_h_Plancherel}
\end{equation}
for almost every $(x,y) \in \R^2$. Multiplying the left side of  
(\ref{G_h_Plancherel}) by $\overline{f(x)}f(y)$ and integrating, we obtain 
\begin{align*}
&(\kappa^2 - \overline{\kappa}^2) \int_{-\infty}^\infty (R(\overline{\kappa}^2)
f)(w) \, \overline{(R(\overline{\kappa}^2) f)(w)} \, dw   \\      
=&
(\kappa^2 - \overline{\kappa}^2) \int_{-\infty}^\infty
(R(\kappa^2) R(\overline{\kappa}^2) f)(w) \, \overline{f(w)} \, dw \\ 
=&
\int_{-\infty}^\infty ((R(\kappa^2) - R(\overline{\kappa}^2))f)(w) 
\, \overline{f(w)} \, dw.
\end{align*}
The interchange of integrals leading to this expression is justified
 because
\begin{eqnarray*}
\int_{-\infty}^\infty \int_{-\infty}^\infty \int_{-\infty}^\infty
\vert G(x,w;\overline{\kappa}^2) \vert \, \vert G(y,w;\overline{\kappa}^2)
\vert \, \vert f(x) \vert \, \vert f(y) \vert \, dxdydw <\infty. \\
\end{eqnarray*}
On the other hand, multiplying the 
right side of (\ref{G_h_Plancherel}) by $\overline{f(x)}f(y)$ and integrating,
we obtain from Lemma \ref{lem:Phi}
$$
\int_{-\infty}^\infty \frac{2i\epsilon}{(\xi^2 - \mu)^2 + \epsilon^2}
\vert \Phi(\xi;\kappa) \vert^2 \, d\xi.
$$
To justify the interchange of integrals leading to this expression
it is necessary to show that 
\begin{equation} 
\int_{-\infty}^\infty \int_{-\infty}^\infty \int_{-\infty}^\infty
\frac{\epsilon}{(\xi^2 - \mu)^2 + \epsilon^2} 
\vert h(x,\xi;\kappa) \vert \, \vert h(y,\xi;\kappa) \vert \, \vert f(x)
\vert \, \vert f(y) \vert \, dxdyd\xi 
\label{h_h_integral}
\end{equation} 
is finite. 
According to the proof of Lemma \ref{lem:F-Green}, $\Vert p(\cdot \, ,\xi;\kappa)
\Vert_{L^2} < c$ independent of $\xi$, for fixed $\kappa$. Using the 
integral equation for $h(x,\xi;\kappa)$ in Lemma \ref{lem:F-Green}
 we see that 
$$
\int_{-\infty}^\infty \vert f(x) \vert \, \vert h(x,\xi;\kappa) \vert \, dx
$$
 is bounded by a constant independent of $\xi$, for fixed $\kappa$.
The finiteness of (\ref{h_h_integral}) follows immediately. 

So far we have 
\begin{equation} 
\int_{-\infty}^\infty ((R(\kappa^2) - R(\overline{\kappa}^2))f)(w) \,
\overline{f(w)} \, dw = \int_{-\infty}^\infty \frac{2i\epsilon}{(\xi^2 
- \mu)^2 + \epsilon^2} \vert \Phi(\xi;\kappa) \vert^2 \, d\xi
\label{big_identity}
\end{equation}
for every $\kappa \in \C$ with ${\Im}\ \kappa > 0$, $\kappa^2
\notin \sigma_{pp}$.  Here $\kappa^2 = \mu + i\epsilon$. Now suppose $\mu,\epsilon > 0$ are chosen
arbitrarily. Then there exists a unique $\kappa \in \C$ in the open ``first octant'' 
$$
\{ \kappa = \rho e^{i \theta} : \rho > 0,\ \theta \in (0,\pi/4) \}
$$
such that $\kappa^2 = \mu + i\epsilon$. Substituting this value of 
$\kappa$ into (\ref{big_identity}) we obtain 
\begin{equation}
\int_{-\infty}^\infty ((R(\mu + i\epsilon) - R(\mu - i\epsilon))f)(w)
\, \overline{f(w)} \, dw = \int_{-\infty}^\infty \frac{2i\epsilon}{(\xi^2
- \mu)^2 + \epsilon^2} \vert \Phi(\xi;\sqrt{\mu + i\epsilon}) 
\vert^2 \, d\xi
\label{big_identity2}
\end{equation} 
where $\sqrt{\mu + i\epsilon}$ is defined to have positive imaginary part. 
Now let $\alpha = x_0 < x_1 < \cdots < x_n = \beta$ 
be a subdivision of $[\alpha,\beta]$ such 
that $x_i - x_{i-1} = (\beta - \alpha)/n$ for each $i=1,\dots,n$. 
According to (\ref{big_identity2})   
\begin{eqnarray}
\lefteqn{ 
\sum_{k=0}^{n-1} \left( \frac{\beta - \alpha}{n} \right) 
\int_{-\infty}^\infty ((R(x_k + i\epsilon)
- R(x_k -i\epsilon))f)(w) \, \overline{f(w)} \, dw } \nonumber \\
&& \ \ \ \ \ \ \ \ = 
\sum_{k=0}^{n-1} \left( \frac{\beta - \alpha}{n} \right)
\int_{-\infty}^\infty \frac{2i\epsilon}{(\xi^2
- x_k)^2 + \epsilon^2} \vert \Phi(\xi;\sqrt{x_k + i\epsilon}) \vert^2 
\, d\xi.      \label{big_identity3}
\end{eqnarray} 
Let $S_{\epsilon,n}$ denote the operator 
$$
\sum_{k=0}^{n-1} \left( \frac{\beta - \alpha}{n} \right) 
(R(x_k + i\epsilon) - R(x_k -i\epsilon)).
 $$
Define the function 
\begin{eqnarray*}
\psi_{\epsilon,n}(x) & = & \sum_{k=0}^{n-1} 
\left( \frac{\beta - \alpha}{n} \right)
((x - x_k - i\epsilon)^{-1} - (x - x_k + i\epsilon)^{-1}) \\ 
& = & \sum_{k=0}^{n-1} \left( \frac{\beta - \alpha}{n} \right)
\frac{2i\epsilon}{(x - x_k)^2 + \epsilon^2}\,. 
\end{eqnarray*}
Then $S_{\epsilon,n} = \psi_{\epsilon,n}(H)$, 
defined by the functional calculus, and we can rewrite the 
left hand side of (\ref{big_identity3}) in the more concise form 
$$
\int_{-\infty}^\infty (\psi_{\epsilon,n}(H)f) \, \overline{f(w)} \, dw.
$$
Define
$$
\psi_\epsilon (x) = \int_\alpha^\beta \frac{2i\epsilon}{(x - \mu)^2 
+ \epsilon^2} \, d\mu.
$$
It is easy to check that $\Vert \psi_\epsilon \Vert_{L^\infty} < \infty$, 
$\sup_n \Vert \psi_{\epsilon,n} \Vert_{L^\infty} < \infty$, and 
$\lim_{n \rightarrow \infty} \Vert \psi_\epsilon - \psi_{\epsilon,n}
\Vert_{L^\infty} = 0$. Therefore, by the functional calculus \cite[Theorem VIII.5]{RS} 
\[
\psi_{\epsilon,n}(H) \rightarrow \psi_\epsilon(H) 
\]
in ${\cal B}(L^2)$ norm as $n \rightarrow \infty$. It follows that 
$$
\lim_{n \rightarrow \infty} \int_{-\infty}^\infty
(\psi_{\epsilon,n}(H)f) \, \overline{f(w)} \, dw
= \int_{-\infty}^\infty (\psi_\epsilon(H)f) \, \overline{f(w)} \, dw.
$$
Further calculation shows that $\sup_{\epsilon > 0} \Vert \psi_\epsilon
\Vert_{L^\infty} < \infty$, and 
\begin{eqnarray*}
\lim_{\epsilon \rightarrow 0^+} \psi_\epsilon (x) & = & 
\left\{ 
\begin{array}{cl}
2 \pi i & \mbox{if $\alpha < x < \beta$}  \\
\pi i  & \mbox{if $x = \alpha$ or $x = \beta$} \\
0      & \mbox{if $x \notin [\alpha,\beta]$} 
\end{array}
\right.   \\ 
& = & \pi i \, (\chi_{[\alpha,\beta]}(x) + \chi_{(\alpha,\beta)}(x)).
\end{eqnarray*} 
Again by the functional calculus, 
$$
\psi_\epsilon(H) f \rightarrow \pi i \, 
(\chi_{[\alpha,\beta]}(H) + \chi_{(\alpha,\beta)}(H))
f = 2\pi i P_{[\alpha,\beta)} f
$$
in $L^2$ norm as $\epsilon \rightarrow 0^+$.
Since $[\sqrt{\al}, \sqrt{\beta}]\cap \cal{E}_0=\emptyset$,
$E=\al$ or $E=\beta\in \sigma_{c}(H)$ (the continuous spectrum) 
 is not an eigenvalue by Lemma \ref{lem:+ve-eigenval}.  
We have  $ P_{[\alpha,\beta]}=P_{(\alpha,\beta)}$. 
 It follows that 
\begin{equation}\label{Eq:Pab(f)}
\lim_{\epsilon \rightarrow 0^+} \int_{-\infty}^\infty
(\psi_\epsilon(H)f) \, \overline{f(w)} \, dw
= 2\pi i \Vert P_{[\alpha,\beta)} f \Vert_{L^2}^2\,. 
\end{equation}

Now we turn to the r.h.s. of (\ref{big_identity3}). Since 
$[ \sqrt{\alpha},\sqrt{\beta} \, ] \cap {\cal E}_0 = \emptyset$, there  
exist numbers $0 < \alpha^\prime < \beta^\prime$ 
and $\epsilon^\prime > 0$ such that:  
$(i)$ $\alpha^\prime < \alpha < \beta < \beta^\prime$; $(ii)$ $[
\sqrt{\alpha^\prime},\sqrt{\beta^\prime} \, ] \cap {\cal E}_0 = \emptyset$; and
$(iii)$ $\sqrt{\mu + i\epsilon}$ is inside the rectangle 
$$
\{ z \in \C : \sqrt{\alpha^\prime} \leq {\Re}\ z \leq 
\sqrt{\beta^\prime},\ 0 \leq { \Im}\ z \leq 1 \}
$$
for all $\mu,\epsilon \geq 0$ such that $\sqrt{\alpha} \leq \sqrt{\mu}
\leq \sqrt{\beta}$, $0 \leq \epsilon \leq \epsilon^\prime$. According
to Lemma \ref{lem:Phi}, $\Phi(\xi;\sqrt{\mu + i\epsilon})$ is bounded and uniformly
continuous in $\xi \in \R$, $\mu,\epsilon \geq 0$ for $\sqrt{\alpha} 
\leq \sqrt{\mu} \leq \sqrt{\beta}$, $0 \leq \epsilon \leq
\epsilon^\prime$. Therefore, for fixed $0 < \epsilon \leq
\epsilon^\prime$, 
\begin{eqnarray*}
\lefteqn{ 
\lim_{n \rightarrow \infty} \sum_{k=0}^{n-1} 
\left( \frac{\beta - \alpha}{n} \right) \int_{-\infty}^\infty 
\frac{2i\epsilon}{(\xi^2 - x_k)^2 + \epsilon^2} \vert \Phi(\xi;
\sqrt{x_k + i\epsilon})\vert^2 \, d\xi } \\
&&\ \ \ \ \ \ \ \ \ \ \ \ \ \ =  
\int_{-\infty}^\infty 
\int_\alpha^\beta \frac{2i\epsilon}{(\xi^2 - \mu)^2 + \epsilon^2}
\vert \Phi(\xi;\sqrt{\mu + i\epsilon})\vert^2 \, d\mu d\xi.
\end{eqnarray*} 
By a standard calculation 
$$
\lim_{\epsilon \rightarrow 0^+} \int_\alpha^\beta 
\frac{2i\epsilon}{(\xi^2 - \mu)^2 + \epsilon^2}
\vert \Phi(\xi;\sqrt{\mu + i\epsilon})\vert^2 \, d\mu
= \left\{ 
\begin{array}{ll} 
2\pi i \vert \Phi(\xi;\vert \xi \vert) \vert^2  & \mbox{if $\xi^2 \in 
(\alpha,\beta)$} \\
0 & \mbox{if $\xi^2 \notin [\alpha,\beta]$} 
\end{array} 
\right. 
$$ 
Another calculation shows that there 
exists a constant $c$ independent of $0 < \epsilon < 1$ such that 
\[
\int_\alpha^\beta \frac{\epsilon}{(\xi^2 - \mu)^2 + \epsilon^2}
\, d\mu < \frac{c}{1 + \xi^4}\,. 
\]
It follows from Lemma \ref{lem:Phi} that 
\begin{equation}\label{ePhi:fsh}
\lim_{\epsilon \rightarrow 0^+} \int_{-\infty}^\infty 
\int_\alpha^\beta \frac{2i\epsilon}{(\xi^2 - \mu)^2 + \epsilon^2}
\vert \Phi(\xi;\sqrt{\mu + i\epsilon})\vert^2 \, d\mu d\xi 
= 2\pi i \int_{\alpha \leq \xi^2 \leq \beta} \vert f^\sharp (\xi) \vert^2 \, d\xi.
\end{equation} 
Therefore, this completes the proof of (\ref{L2:Pf:Ff}) by comparing \eqref{Eq:Pab(f)} and (\ref{ePhi:fsh}).
\end{proof}



\bigskip 

We are now ready to prove the main theorem. 

\begin{proof}[Proof of Theorem \ref{th:Plancherel}]  Let $\cal{O} = (0,\infty) \setminus \cal{ E}_0^2$. Recall that 
$\cal{ E}_0$ is bounded and $\cE\cup \{0\}$ is a closed set of measure zero. 
 Hence the set $\cal{O}$ can be expressed as a countable disjoint union 
\[ \cal{ O} = \bigcup_{i = 1}^\infty [\alpha_i,{\beta_i})  \]
with each $[\sqrt{\alpha_i},\sqrt{\beta_i}] \cap \cal{ E}_0 = \emptyset$. 
 Therefore, for every $f \in L^2$  
\begin{align*}
f =& P_{\cal{ O}}f+ P_{\cal{ E}_0^2}f+
P_{\sigma_{pp}}f\\
 =& \lim_{N \rightarrow \infty} \sum_{i=1}^N 
P_{[\alpha_i,\beta_i)}f +P_{pp}f +
P_{sc}f
\end{align*}
with convergence in the $L^2$ sense and $P_{sc}f=0$.
In this sum the functions $P_{[\alpha_i,\beta_i)}f$ 
are mutually orthogonal. So, 
according to Lemma \ref{lem:Pf}, for all $\phi \in C^\infty_0$, 
\begin{align}
\Vert P_{ac} \phi \Vert_{L^2}^2 =& \lim_{N \rightarrow \infty} \sum_{i=1}^N
\Vert P_{[\alpha_i,\beta_i)}\phi \Vert_{L^2}^2 
= \lim_{N \rightarrow \infty} \sum_{i=1}^N \int_{\alpha_i \leq \xi^2  
\leq \beta_i} \vert \phi^\sharp (\xi) \vert^2 \, d\xi  = \Vert \phi^\sharp \Vert_{L^2}^2 \label{1.1_identity1}\\
\Vert P_{pp} \phi \Vert_{L^2}^2=& \sum_{\lam_k}|\la \phi,e_k\ra|^2  \, \notag
\end{align} 
where $e_k$ are eigenfunctions of $\lam_k\in \sigma_{pp}(H)$ and $\{e_k\}$ can be chosen to constitutes an orthonormal basis of $\sH_{pp}$. 
\begin{enumerate}
\item[(i)] Proof of Theorem \ref{th:Plancherel} ({\em a}) and ({\em c}) is in fact an extension of (\ref{1.1_identity1}) to the square integrable functions in $L^2$. 

Now suppose $f \in L^2$ vanishes outside $[-M,M]$ first. 
According to Lemma \ref{l:e(xxi):E0} the integral 
\[
\fsh(\xi) = (2\pi)^{-1/2} \int_{-\infty}^\infty f(x) \overline{e(x,\xi)} \, dx \] 
converges absolutely for a.e. $\xi$. We claim that $\fsh \in L^2$ and $\Vert {\fsh} \Vert_{L^2} = \Vert P_{ac}f \Vert_{L^2}$. 
Let $\varphi_n \in C^\infty_0$ be a sequence of functions such that: 
(a) $\varphi_n \rightarrow f$ in $L^2$ as $n \rightarrow \infty$; 
and (b) each 
$\varphi_n$ vanishes outside $[-2M,2M]$. Again using Lemma \ref{l:e(xxi):E0} we see that 
 \[
\lim_{n \rightarrow \infty} {\varphi}_n^\sharp(\xi) = \lim_{n \rightarrow \infty}
(2 \pi)^{-1/2} \int_{-2M}^{2M} \varphi_n(x) \overline{e(x,\xi)} \, dx  = \fsh(\xi)
\]
for a.e. $\xi$. On the other hand, according to (\ref{1.1_identity1}) the sequence of functions ${\varphi}_n^\sharp$ is a Cauchy sequence in $L^2$ with
some limit $g \in L^2$. The only possibility is that $g = {\fsh}$ pointwise almost everywhere, and  
\begin{equation} 
\Vert {\fsh} \Vert_{L^2} = \lim_{n \rightarrow \infty} \Vert {\varphi}_n^\sharp
\Vert_{L^2} = \lim_{n \rightarrow \infty} \Vert P_{ac}\varphi_n \Vert_{L^2} 
= \Vert P_{ac}f \Vert_{L^2}.
\label{1.1_identity2}
\end{equation}

Now we consider the case of general $f \in L^2$. 
For each $M = 1,2,\dots$
let $f_M(x) = \chi_{[-n,n]}(x) f(x)$. Of course $f_M \rightarrow f$
in $L^2$, so according to (\ref{1.1_identity2}) the sequence  $f_M^\sharp$ is Cauchy in $L^2$.
{\em Let ${\fsh}$ denote the limit}. Since $\Vert f_M^\sharp \Vert_{L^2} = \Vert P_{ac}f_M \Vert_{L^2}$ for all $M$, $\Vert \fsh \Vert_{L^2}
= \Vert P_{ac} f \Vert_{L^2}$. This proves {\em (a)} and {\em (c)}.

\item[(ii)] To prove {\em (b)} it suffices to show that 
\[
P_{[\alpha_i,\beta_i)} f(x) = (2\pi)^{-1/2} \int_{\alpha_i \leq \xi^2 \leq \beta_i} \fsh (\xi) e(x,\xi) \, d\xi \] 
in  $L^2$ and a.e., for each of the intervals $[\sqrt{\alpha_i},\sqrt{\beta_i})$ in the decomposition of $\cal{ O}$. Note that statement {\em (c)} in the theorem implies that the conclusion of Lemma \ref{lem:Pf} holds for arbitrary $f \in L^2$. It follows by polarization that 
\begin{align*}
\int_{-\infty}^\infty P_{[\alpha_i,\beta_i)} f(x) \, \overline{g(x)} \, dx 
= \int_{\alpha_i \leq \xi^2 \leq \beta_i} {f}^\sharp(\xi) 
\overline{{g}^\sharp(\xi)}  \, d\xi 
\end{align*}
for all $f,g \in L^2$. In particular, if $f \in L^2$ and $g \in C^\infty_0$, then 
\begin{align*}
\int_{-\infty}^\infty P_{[\alpha_i,\beta_i)} f(x) \, \overline{g(x)} \, dx
 =& (2\pi)^{-1/2} \int_{\alpha_i \leq \xi^2 \leq \beta_i}  {f}^\sharp(\xi) \left( \int_{-\infty}^\infty 
\overline{g(x)} e(x,\xi) \, dx \right) \, d\xi \\
 =& (2\pi)^{-1/2} \int_{-\infty}^\infty \left( \int_{\alpha_i \leq \xi^2 \leq \beta_i} {f}^\sharp(\xi) e(x,\xi) \, d\xi \right) \, 
 \overline{g(x)} \, dx\,,
\end{align*}
where we note that $\vert e(x,\xi)\vert\le c_{\al,\beta}$ for all $(x,\xi)\in \R\times[\sqrt{\al},\sqrt{\beta}] $. 
 by Remark \ref{e(xxi):alpha-beta}. 
  This proves {\em (b)}.  

\item[(iii)] Proof of (\emph{d}).  If $\phi \in C^\infty_0$,  then $(d)$ holds for $\phi$  %
for $|\xi| \notin \cal{ E}_0 \cup \{0\}$  since $e(\cdot,\xi)$ is a weak solution of (\ref{S_eq}), 
which follows from the fact that $Ve(\cdot,\xi)\in L^1$  
in view of (\ref{eq:Ve}) and  the remark prior to  (\ref{LS_mod_eq}). 

Now suppose $f \in \cal{ D}(H) $ and let $\varphi_n \in C^\infty_0$ be a
sequence of functions such that $\varphi_n \rightarrow f$ in $L^2$ and
$H \varphi_n \rightarrow Hf$ in $L^2$. We can do so because  $C_0^\iy$ is a core of $\cal{ D}(H)$.
  By passing to a subsequence if
necessary, we may assume that ${\varphi}^\sharp_n(\xi) \rightarrow {f}^\sharp(\xi)$
and $(H\varphi_n)^\sharp {(\xi)} \rightarrow (Hf)^\sharp {(\xi)}$ pointwise for
almost every $\xi$. Since $(H\varphi_n)^\sharp{(\xi)} = \xi^2
{\varphi}^\sharp_n(\xi)$,  it follows that $(Hf)^\sharp {(\xi)} = \xi^2 {f}^\sharp(\xi)$  for almost every $\xi$ and hence in $L^2$.   
This concludes the proof of Theorem \ref{th:Plancherel}. 
\end{enumerate}
\end{proof}  

\section{Wave operators and surjectivity of $\cF$}\label{s:Om:F:surj}

 In Theorem \ref{th:Plancherel} we have defined the perturbed Fourier transform $\cF$ in (\ref{cFf:L2:V}) and proven the first identity in  (\ref{Eac:F*F}). 
In this section we turn to proving the other identity $\cF\cF^*=I$ in  (\ref{Eac:F*F}), 
see Proposition \ref{c:wave:F*F0}. 
To do this we first need some fundamental relation between the wave operator $\Om:=\Om_-$ and  $\mathcal{F}$ 
in \eqref{surjection_equation} of Lemma \ref{l:Om:F0F}. 
Then we will  show $\cF$ is a surjection onto $L^2$ so that its inverse is indeed given by $\cF^*$ defined as (\ref{eF*:inv}). 

\subsection{Wave operator $\Om$ and $\cF$ }\label{ss:Om:cF}
 For $f \in L^2$ define the wave operators $\Om_\mp$ to be 
\begin{align}\label{e:Omega-f}
&\Om_\mp f: = \lim_{t \rightarrow \pm\infty} e^{-itH} e^{itH_0} f\,. 
\end{align}
According to Lemma \ref{l:exist:ac}, 
 $\Omega_\mp$ exist and are  well-defined isometries on $L^2$. 
Hence  
 their adjoints $\Omega_\mp^*$ are  surjections on $L^2$ and  $\Omega_\mp^* \Omega_\mp =I$. 

\begin{lemma}\label{l:Om:F0F}  For all $f \in L^2$, we have 
\begin{equation}
\widehat{\Omega_-^* f} = f^\sharp\,,\label{surjection_equation}
\end{equation}
equivalently, 
\begin{equation}
{\Omega_- f} = \cF^*\cF_0 f \,, \label{eq:Om:F*F0}
\end{equation}
where $\cF f=\fsh$ is defined as  (\ref{cFf:L2:V}), $\cF^*$ is the adjoint and
$\cal{F}_0f =\hat{f}$ denotes the ordinary Fourier transform  (\ref{F0:fourier}) on $L^2$. 
\end{lemma}
We postpone  the proof of {Lemma \ref{l:Om:F0F}} till the end of this subsection.  

\bigskip

\begin{prop}\label{Th:F-surj} Suppse $V\in L^1\cap L^2(\R)$ is real-valued. Then the mapping $f \mapsto 
f^\sharp$ is a surjection on $L^2(\R)$. Therefore,
\begin{enumerate}
\item[(a)] ${\cal F}$: $\sH_{ac}\to L^2$ is 
a bijective isometry. 
\item[(b)] $\ran \Om_-=\sH_{ac}$\,.
\end{enumerate}
\end{prop}

\begin{proof} 
Writing $\Om=\Om_-$, equation (\ref{surjection_equation}) reads $\mathcal{F}=\mathcal{F}_0\Om^*$.  
Substituting $f$ with  $\Om f$ in (\ref{surjection_equation})
 gives 
$(\Omega f)^\sharp =(\Omega^* \Omega f)^\wedge =\hat{f}$. 
 This suffices to prove the proposition since  $\cal{F}_0$ is a surjection on $L^2$.
 In view of (\ref{eq:Om:F*F0})  and (\ref{eF:Eac}), 
 we have $\Om=\cal{F}^{*}\cal{F}_0$  and $\cal{F}^{*}\cF= P_{ac}$, it follows that 
    $\textrm{Ran}\, \Om=\sH_{ac}$ 
 whence we know 
$\cal{F}: L^2\rightarrow L^2 $ is surjective. 
\end{proof}

The identities we have shown in Lemma \ref{l:Om:F0F} and Proposition \ref{Th:F-surj} also hold true for $\Om_+$ and $\cF_-$
in a parallel pattern, see Remark \ref{re:E(xxi)pm:LS}. Thus we  obtain some useful decompositions for the wave operators. 
In particular, we can prove the second identity $\cF\cF^*=I$ in (\ref{Eac:F*F}), recalling the notation $\cF=\cF_+$. 

\begin{prop}\label{c:wave:F*F0} Under the condition of $V$ in Proposition \ref{Th:F-surj}, we have
\begin{align} 
&\Omega_-  = \cF^*_+ \cF_0\,,  \quad \Omega_+  = \cF^*_- \cF_0    \label{eOm:F*+}\\
& \Om_\mp^*\Om_\mp=I, \quad  \Om_\mp\Om_\mp^*=P_{ac}\,   \label{Om:Pac}\\ 
&  \cF_\pm  \cF_\pm^*=I, \quad  \cF_\pm^*  \cF_\pm=P_{ac}\,, \label{FFstar:Id} 
\end{align}
where $\cF^*_\pm$ are the perturbed Fourier transforms for $H_V$. 
\end{prop} 

\begin{proof}  We only  show the case for $\Om=\Om_-$.  The other equations  involving $\Om_+$ can be proven similarly. 
First,  (\ref{eOm:F*+}) is simply (\ref{eq:Om:F*F0}). 
Secondly, since $\Om$ is an isometry, $\Om^*\Om=I$.   Using (\ref{eq:Om:F*F0}) we obtain
\begin{align} 
 \cF_0^* \cF_+ \cF_+^*\cF_0 =I\,.
\end{align} 
 It follows immediately that  $ \cF_+  \cF_+^*=I$ which proves (\ref{FFstar:Id}).  
Thirdly,  $\Om\Om^*=(\Om\Om^*)^2$ is a projection, in fact, 
\begin{align*}
& \Om\Om^*= \cF^*_+ \cF_0\cF_0^*\cF_+ =\cF^*_+ \cF_+=P_{ac}\,.   
\end{align*}  
This proves (\ref{Om:Pac}).   
\end{proof}
\begin{remark}\label{rem:Eac-Om:W} The proof of (\ref{Om:Pac}) shows that $\ran\Om=\sH_{ac} $ if and only if
the inversion formula holds $\cF^*\cF=I$ on $L^2$,  that is, asymptotic completeness of $\Om$ implies $\sigma_{sc}(H)=\emptyset$.
 The identities in the proposition also hold for $V$ a finite measure satisfying (\ref{eV:wei2}) 
 as is stated in Theorem \ref{Fo:waV}. 
\end{remark}

Next we prove some simple  consequences. 
\begin{corollary}\label{Cor:property}  Let $V\in L^1\cap L^2$.  Then $\cF$ has the following properties. 
\begin{enumerate}
\item[(a)] If $f \in L^2$ and $\xi^2 f^\sharp(\xi) \in L^2$, 
then $f \in {\cal D}(H)$. 
\item[(b)] If $f \in L^2$ and $0 < \alpha < \beta$ are such that  $[\sqrt{\alpha},\sqrt{\beta}] \cap {\cal E}_0 = \emptyset$, then
\begin{enumerate}
\item[(i)] $(P_{[\alpha,\beta)}f)^\sharp(\xi) = \chi_{\{\alpha \leq \xi^2 \leq \beta\}}
f^\sharp(\xi)$. 
\item[(ii)] 
$P_{[\alpha,\beta)}f \in {\cal D}(H^n)$ for all $n=1,2,3,\dots$.  
\end{enumerate}\end{enumerate}
\end{corollary} 

\begin{proof} Since $H$ is self-adjoint, $f \in {\cal D}(H)$ if and only if there 
exists an element $h \in L^2$ such that $\int_{-\infty}^\infty 
f(x) \overline{Hg(x)} \, dx = 
\int_{-\infty}^\infty h(x) \overline{g(x)} \, dx$
for all $g \in {\cal D}(H)$. Now suppose $f \in L^2$ and $\xi^2 f^\sharp(\xi)
\in L^2$. According to Theorem \ref{th:Plancherel}, if $g \in {\cal D}(H)$, then 
$\int_{-\infty}^\infty f(x) \overline{Hg(x)} \, dx = 
\int_{-\infty}^\infty f^\sharp(\xi) \overline{\xi^2 g^\sharp(\xi)} \, d\xi$.
But according to Proposition \ref{Th:F-surj} there exists an element $h \in L^2$ such that
$h^\sharp(\xi) = \xi^2 f^\sharp(\xi)$. Statement $(a)$ follows.

Now let $f,g \in L^2$. According to Proposition \ref{Th:F-surj}, $g(\xi) = h^\sharp(\xi)$ 
for some $h \in L^2$. From 
Theorem \ref{th:Plancherel} (c),  Lemma \ref{lem:Pf}  and polarization we  have 
\begin{eqnarray*} 
\int_{-\infty}^\infty g(\xi) \overline{(P_{[\alpha,\beta)} f)^\sharp
(\xi)} \, d\xi & = & 
\int_{-\infty}^\infty h^\sharp(\xi) \overline{(P_{[\alpha,\beta)} f)^\sharp
(\xi)} \, d\xi \\ 
& = & \int_{-\infty}^\infty h(x) \overline{P_{[\alpha,\beta)} f(x)} \, dx \\
& = & \int_{-\infty}^\infty P_{[\alpha,\beta)}h(x) \overline{P_{[\alpha,\beta)}
f(x)} \, dx \\ 
& = & \int_{\alpha \leq \xi^2 \leq \beta} h^\sharp(\xi) \overline{f^\sharp(\xi)}
\, d\xi \\
& = & \int_{\alpha \leq \xi^2 \leq \beta} g(\xi) \overline{f^\sharp(\xi)}
\, d\xi.
\end{eqnarray*} 
Since $g$ is arbitrary, this proves first statement of $(b)$. Finally, the second statement of $(b)$ 
follows from $(a)$ and $(b)$({\em i}), or directly from the  spectral theorem.  
\end{proof}

\begin{remark} We wish to comment that technically the proofs of Proposition \ref{Th:F-surj}
and Lemma \ref{lem:Pf} are independent. The former relies on the existence of the wave operator 
while the latter relies on the construction of scattering solutions and the resolvent estimation in Lemmas \ref{L:R(z)} to \ref{lem:Phi}.  
\end{remark}


\bigskip
\begin{proof}[Proof of Lemma \ref{l:Om:F0F}] To prove (\ref{surjection_equation}) it suffices to show that 
\[
\int_{-\infty}^\infty \Omega g(x) \overline{f(x)} \, dx 
= \int_{-\infty}^\infty \hat{g}(\xi) \overline{f^\sharp(\xi)} \, d\xi  \] 
for a dense set of $g \in L^2$ and a dense set of $f \in L^2$.  In dong so, we only {need} to show that 
\begin{equation} 
\int_{-\infty}^\infty \Omega g(x) \overline{P_{[\alpha,\beta]}f(x)} \, dx
= \int_{-\infty}^\infty \hat{g}(\xi) \overline{(P_{[\alpha,\beta]}f)^\sharp
(\xi)} \, d\xi
\label{surjection_equation2}
\end{equation}
for every $g \in C^\infty_0$, $f \in L^2$, $0 < \alpha < \beta$ with
$[\sqrt{\alpha},\sqrt{\beta}] \cap {\cal E}_0= \emptyset$. 

So fix $g$, $f$, $\alpha,
\beta$ in this way, and define $h = P_{[\alpha,\beta]}f$ (also see  Corollary \ref{Cor:property}). 
The hypotheses on $g$ and $h$ justify all the interchanges of integrals and limits that follow. 
By a simple calculation 
\[
\int_{-\infty}^\infty \Omega g(x) \overline{h(x)} \, dx
- \int_{-\infty}^\infty g(x) \overline{h(x)} \, dx
= -i \lim_{t \rightarrow \infty} \int_0^t \int_{-\infty}^\infty
(e^{-isH} V e^{isH_0}g)(x) \overline{h(x)} \, dx \, ds
\] 
where we use $ \frac{d}{dt} (e^{-i t H} e^{i t H_0} ) = -i  e^{-i t
H} V e^{i t H_0}$ and apply Fubini for every $t>0$.   

Thus, by \cite[XI.6, Lemma 5]{RS},  (the lemma on ``abelian limits'') 
\begin{equation}
\begin{aligned}
\mbox{}&\int_{-\infty}^\infty \Omega g(x) \overline{h(x)} \, dx \label{surjection_equation3}  \\ 
 =& \int_{-\infty}^\infty g(x) \overline{h(x)} \, dx
\ -\ i \lim_{\epsilon \rightarrow 0^+} \int_0^\infty 
\int_{-\infty}^\infty (e^{-isH} V e^{isH_0}g)(x) \overline{h(x)}
e^{-\epsilon s} \, dx \, ds.  
\end{aligned}    
\end{equation}

Now using Theorem \ref{th:Plancherel} and the spectral theorem it is easy to show that
if $k \in \sH_{ac}$, then 
$(e^{itH} k)^\sharp(\xi) = e^{it\xi^2} k^\sharp(\xi)$. Consequently 
\begin{eqnarray*}
\lefteqn{
\int_{-\infty}^\infty (e^{-isH} V e^{isH_0}g)(x) \overline{h(x)} \, dx } \\
&& \ \ \ \ = \int_{-\infty}^\infty e^{-is\xi^2} (V e^{isH_0}g)^\sharp(\xi) 
\overline{h^\sharp (\xi)} \, d\xi \\ 
&& \ \ \ \ = (2\pi)^{-1/2} \int_{-\infty}^\infty \int_{-\infty}^\infty 
e^{-is\xi^2} V(x) (e^{isH_0}g)(x) \overline{e(x,\xi)} \, \overline{h^\sharp
(\xi)} \, dx \, d\xi.
\end{eqnarray*} 
Thus  
\begin{eqnarray*}
\lefteqn{
\int_0^\infty \int_{-\infty}^\infty 
(e^{-isH} V e^{isH_0}g)(x) \overline{h(x)}
e^{-\epsilon s} \, dx \, ds } \\
&& = (2\pi)^{-1/2} \int_0^\infty \int_{-\infty}^\infty \int_{-\infty}^\infty
V(x) (e^{is(H_0 - \xi^2 + i\epsilon)} g)(x) 
\overline{e(x,\xi)} \, \overline{h^\sharp(\xi)} \, dx \, d\xi \, ds \\
&& = i (2\pi)^{-1/2} \int_{-\infty}^\infty \int_{-\infty}^\infty
V(x) ((H_0 - \xi^2 + i\epsilon)^{-1}g)(x) 
\overline{e(x,\xi)} \, \overline{h^\sharp(\xi)} \, dx \, d\xi.
\end{eqnarray*}

Using $\overline{G_0(x,y;z)} = G_0(x,y;\overline{z})$ we have
\begin{eqnarray*} 
\lefteqn{ 
(H_0 - \xi^2 + i\epsilon)^{-1}g(x)  } \\ 
&& \ \ \ \ \ \ \ \ \ \ =  
\int_{-\infty}^\infty G_0(x,y;\xi^2 - i\epsilon) g(y) \, dy \\
&& \ \ \ \ \ \ \ \ \ \ = 
(2i \overline{(\xi^2 + i\epsilon)^{1/2}})^{-1} 
\int_{-\infty}^\infty 
e^{-i \overline{(\xi^2 + i\epsilon)^{1/2}} \vert x-y \vert}
g(y) \, dy.
\end{eqnarray*}
Substituting all of this into (\ref{surjection_equation3}) we now have
\begin{align*}
\mbox{}&
\int_{-\infty}^\infty \Omega g(x) \overline{h(x)} \, dx = \int_{-\infty}^\infty g(x) \overline{h(x)} \, dx  \\  
+& (2\pi)^{-1/2} \int_{-\infty}^\infty \int_{-\infty}^\infty
\int_{-\infty}^\infty V(x) (2i\vert \xi \vert)^{-1} e^{-i \vert \xi \vert
\vert x-y \vert} g(y) \overline{e(x,\xi)} \, \overline{h^\sharp(\xi)}
\, dy \, dx \, d\xi \\ 
 =& \int_{-\infty}^\infty g(x) \overline{h(x)} \, dx
+ (2\pi)^{-1/2} \int_{-\infty}^\infty \int_{-\infty}^\infty 
g(y) (e^{-iy\xi} - \overline{e(y,\xi)}) \overline{h^\sharp(\xi)} \, dy \, d\xi \\
 =& \int_{-\infty}^\infty \hat{g}(x) \overline{h^\sharp(x)} \, dx\,.
\end{align*}
Here for the first equality we have applied  the dominated convergence theorem since $ e(x, \xi)$ is uniformly bounded for all $(x,\xi)\in \R\times [\sqrt{\alpha}, \sqrt{\beta}]$ and
$h^\sharp (\xi) = \chi_{\{\alpha \le \xi^2 \le \beta\}} f^\sharp \in L^2\cap L^1$  by Corollary \ref{Cor:property} ({\em b}). 
To arrive at the second to last line we used equation (\ref{LS_eq}).   
\end{proof}


\subsection{Integral representation for spectral operators}\label{ss:phi(H):kernel}  
 If $\vphi \in C_b=C_b^0(\R)$, 
 define a bounded operator on $L^2$ by (\ref{def:specOp}),
where $\cF$, $\cF^*$ and $\{e_k\}$ are the same as in Theorem \ref{th:Plancherel}.  
Then in light of Proposition \ref{c:wave:F*F0},  
it is easy to verify $\vphi(H)$ satisfies all the properties of the spectral theorem.  
Hence by uniqueness, our definition agrees with the definition of a spectral operator using functional calculus.  
We derive an integral kernel of  $\vphi(H)$.  
\begin{prop}\label{c:vphi(H):ker}    
Suppose  $V\in L^1\cap L^2$.  Let $\varphi : \R \rightarrow \C$ be continuous and supported in $ [\alpha,\beta]$, where $0 < \alpha < \beta$, $[\sqrt{\alpha},
\sqrt{\beta}] \cap {\cal E}_0 = \emptyset$. Then for all $f \in L^1 \cap L^2$
$$
\varphi(H)f(x) = \int_{-\infty}^\infty K(x,y) f(y) \, dy 
$$
where  
\[ K(x,y) = (2\pi)^{-1} \int_{-\infty}^\infty \varphi(\xi^2) e(x,\xi)
\overline{e(y,\xi)} \, d\xi\,. \]
\end{prop}

\begin{proof} Let $p_1,p_2, {p_3},\dots$ be a sequence of complex polynomials such that
$p_n \rightarrow \varphi$ uniformly on $[\alpha,\beta]$. From the
spectral theorem we have 
$$
\lim_{n \rightarrow \infty} \Vert \varphi(H) f - p_n(H) P_{[\alpha,\beta]} f
\Vert_{L^2} = 0. 
$$
Hence  
$$
\lim_{n \rightarrow \infty} \int_{-\infty}^\infty
\vert (\varphi(H) f)^\sharp(\xi) -  (p_n(H) P_{[\alpha,\beta]} f)^\sharp
(\xi) \vert^2 \, d\xi = 0.
$$
According to Corollary \ref{Cor:property} this implies that  
$$
\lim_{n \rightarrow \infty} \int_{-\infty}^\infty
\vert (\varphi(H) f)^\sharp(\xi) - 
p_n(\xi^2) \chi_{\alpha < \xi^2 < \beta} f^\sharp(\xi) \vert^2 \, d\xi = 0.
$$
It follows that $(\varphi(H)f)^\sharp(\xi) = \varphi(\xi^2) f^\sharp(\xi)$ 
in the $L^2$ sense. 
Because of the support condition on $\varphi$ 
we now have 
$$
\varphi(H)f (x) = (2\pi)^{-1/2} \int_{-\infty}^\infty \varphi(\xi^2)
f^\sharp(\xi) e(x,\xi) \, d\xi.
$$ 
Again because of the condition that $[\sqrt{\alpha},\sqrt{\beta}] 
\cap {\cal E}_0 = \emptyset$, there exists some constant $c < \infty$
such that $\vert e(x,\xi) \vert < c$ for all $x \in \R$, $\alpha \leq \xi^2
\leq \beta$, see  Remark \ref{e(xxi):alpha-beta}. 
For $f \in L^1 \cap L^2$, the proof is finished with an application of Fubini's theorem. 
\end{proof}

\begin{remark}\label{rem:E(xxi)} 
  Owing to the fact that $\sigma_{ac}(H)=[0,\iy)$ and $\cE_0$ has zero measure with $\cE_0\cup \{0\}$ being closed, 
  we observe that  
one can allow $0\le \alpha < \beta $ be arbitrary. Indeed, a simple density argument  yields
\begin{equation}\label{ecFf:vphi}
\cF(\vphi(H) f) =  \vphi(\xi^2) \cF f\,,
\end{equation}
in view of 
\begin{align*}
&\varphi_{\cup_i [\al_i, \beta_i)}(H) f {\rightarrow} \vphi(H) f \;\;  \qquad \textrm{in} \;L^2\\  
& \varphi_{\cup_i [\al_i, \beta_i)}(\xi^2) \fsh {\rightarrow} \vphi(\xi^2) f^\sharp \ \qquad \textrm{in} \;L^2\,
\end{align*}
where one can always choose $\cup_i (\al_i, \beta_i)=(0,\iy)\setminus\cE_0^2$ such that each $[\al_i,\beta_i]$ is disjoint from  {$\cE_0^2=\{\xi^2:\xi\in\mathcal{E}_0\}$}. 
\end{remark}


\begin{prop}\label{p:phi(H)(xy)} Let $V\in L^1\cap L^2$ be real-valued. 
Let $\vphi$ be continuous and have compact support. Suppose $e(x,\xi)$ is uniformly bounded for a.e.
$(x,\xi)\in \R^2$. Then for all $f\in \sH_{ac}$\,,
\begin{align}\label{e:K(xy)}
\vphi(H)f(x) = \int K(x,y)f(y)dy
\end{align}
where $K(x,y)= \int \vphi(\xi^2) e(x,\xi) \overline{e(y,\xi)}d\xi $. 
The integral for the kernel $K$ converges for a.e. $x, y$ such that
$K(x,\cdot)$ is in $L^2\cap L^\iy$  for a.e. $x$.
\end{prop}

\begin{proof} Suppose $f\in C_0^\infty$ first, which implies $f^\sharp ={\cal F}f \in L^2\cap L^\infty$. According to (\ref{ecFf:vphi}),
\begin{align*} 
\vphi(H)f(x)\stackrel{L^2}{=}& \slim \frac{1}{\sqrt{2\pi}}\int^N_{-N} \vphi(\xi^2)
f^\sharp (\xi) e(x,\xi)d\xi\\ 
=&\frac{1}{\sqrt{2\pi}}\int \underbrace{\vphi(\xi^2)
f^\sharp (\xi)}_{L^2\cap L^1} e(x,\xi)d\xi\qquad (pointwise)\\ 
\end{align*}  

Applying Fubini theorem we have 
\[ \vphi(H)f(x)=\frac{1}{\sqrt{2\pi}}\int f(y)dy \int \vphi(\xi^2) e(x,\xi) \overline{e(y,\xi)} d\xi,
 \]
where the integral is convergent for a.e. $x$. 

Now let $f\in L^2$. Pick a sequence $\{f_m \}\subset C_0^\infty$
such that $f_m \rightarrow f$ in $L^2$ and $ \vphi(H)f_m
\rightarrow\vphi(H)f$\quad $a.e$ and in $L^2$.  
Thus, since for almost all $x$, $K(x,\cdot) $ is in $L^2$,
\begin{align*}
\vphi(H)f(x)=&\lim_{m\to \iy} \vphi(H)f_m(x)\quad a.e.\\
 =&\lim_{m\to \iy} \int \vphi(H)(x,y) f_m dy = \int \vphi(H)(x,y) f dy. 
\end{align*} 
The fact that  
$K(x,y):=\vphi(H)(x,\cdot) \in L^2\cap L^\infty$ 
is a general property for Carleman type kernels \cite[Corollary A.1.2]{Sim82}, which can also be easily verified
using the heat kernel estimate \cite{OZh08,Zheng2006}.  
 \end{proof}


\begin{remark}\label{rem:V:exp-dec} If  $Ve^{\al |x|}\in L^1(\R)$, $\al>0$, then $e(x,\xi)$ are uniformly bounded and
the set  $\mathcal{E}_0$ is empty  
by applying  the analytic Fredholm theorem in the region $\Im z> \al/2 $ in the proof of Lemma \ref{l:e(xxi):E0},  
cf. \cite{Sim71} also. 
\end{remark}

\begin{remark}\label{rem:kernel:V1d} 
  Under the weaker condition $\sup_x\int_{|x-y|\le 1} |V(y)|dy< \iy$, there is a similar result  as stated in \cite[Theorem C.5.4]{Sim82}. 
  However, the spectral measure involved was not explicitly given there.  
 In our setting, one can derive the spectral measure from (\ref{dE:density-meausre}), 
{which shows the density of $dE_\lam$ for almost all $\lam$. 
This is comparable to (\ref{Hv:spectra-measure}) in the case of $V$ being a short-range finite measure}. 
\end{remark}

\section{Wave operators and scattering of $e^{-itH}$}\label{s:waveOp:scatter} 
In this section we  prove a scattering result for $U(t)=e^{-itH}$ under the condition that $V$ is real-valued satisfying either   
 \begin{enumerate}
\item[(i)]  $\cM_2$, i.e., a finite  measure verifying (\ref{eV:wei2}),  (see Theorem \ref{t:wave:scatter});  or,  
 \item[(ii)]  $L^1\cap L^2$,  (see Theorem \ref{th:Omega}).
\end{enumerate}
Let $\vphi(H)=\int \vphi(\lam) dE_\lam$ be defined by means of the spectral resolutions. 
The following  intertwining property for $\vphi(H)$ and $\vphi(H_0)$ are shown 
in \cite[sections {VIII and} IX]
{GH98} via the perturbed Fourier transforms.
\begin{prop}\label{p:phi(H):H0}  Let  $H=H_0+V$ be defined as a form operator as in \S \ref{s:preliminV}. 
Denote by  $\Fo_0f$   
  the usual Fourier transform \eqref{F0:fourier}
and   $ \Fo_\pm, \Fo^*_\pm$  the perturbed Fourier transforms and their adjoints (\ref{Fo+:f})-(\ref{eFo*-g}).  
If $\vphi$ is any bounded measurable function on $\R$, then we have
\begin{align}\label{Fstar:vphi}
\vphi(H)  \Fo^*_\pm \Fo_0= \Fo^*_\pm  \Fo_0\vphi(H_0)
\end{align}
or, equivalently, 
\begin{align}\label{eW:intertwine}
&\vphi(H) P_{ac} =W_\mp \vphi(H_0)W_\mp^*\,.
\end{align}
\end{prop}

\begin{proof} 
For $u=\psi^\pm_k$, $\lam=k^2>0$,  using
$\la (-\De+V)u, \phi\ra=\lam\la  u,  \phi\ra$ for all $\phi\in \cD$,  
Fubini theorem and  definitions of $\Fo^*_\mp$, we obtain for $\zeta\in \C\setminus (-\iy,\iy)$ and any $f\in \cD(H_0)=W^{2,2}$
\begin{equation}\label{Psi(k):cE0}
\begin{aligned}
&(H-\zeta)   \Fo^*_\pm \Fo_0 f =  \Fo^*_\pm \Fo_0 (H_0-\zeta) f\,\notag\\
=& \frac1{\sqrt{2\pi}}\int (k^2-\zeta)\psi^\pm (x,k)\hat{f}(k)dk\,, 
\end{aligned}
\end{equation} which implies that
\begin{align}
&(H-\zeta)\inv    \Fo^*_\pm \Fo_0 g =  \Fo^*_\pm \Fo_0 (H_0-\zeta)\inv g\,,\quad \forall g\in L^2\,.   \label{R(z):FoF0:g}
\end{align} 
Then, Stone's formula \eqref{Eab:R(zeta)} gives 
\begin{align*}
&\phi(H)  \Fo^*_\pm \Fo_0 g =  \Fo^*_\pm \Fo_0 \phi(H_0)g\,,
\end{align*} 
which amounts to (\ref{eW:intertwine}) in view of  \eqref{eWav:pm} and \eqref{ac:WWstar}. 
\end{proof} 

\begin{remark}\label{spectra-resolution:W}  By virtue of Theorem \ref{Fo:waV}, it is easy to note that 
equation (\ref{eW:intertwine}) is equivalent {to} (\ref{spec:F*Fo}). 
\end{remark}


We are ready to prove a long time representation of $U(t)$ in terms of $W_\pm$ and $U_0(t)$. 
\begin{theorem}\label{t:wave:scatter} Under the assumption  (\ref{eV:wei2}) on $V$,  the solution to (\ref{eU:H_V}) scatters:  
\begin{enumerate}
\item[(a)] If $\phi\in \sH_{ac}\cap L^1_1$, then there exists $\eta_\mp\in \sH_{ac}$ such that 
\begin{align}\label{U(t):V-W} 
&U(t)\phi = W_\mp U_0(t)\phi  +U(t)\eta_\mp\quad 
\end{align} 
where $\eta_\mp$ satisfy 
\begin{align*}
&\left\vert\Fo_\pm (\eta_\mp)(k)\right\vert 
\le \frac{C } {1+|k|}\norm{V}_{L^1_2} \norm{\phi}_{L^1_1}\,,
\end{align*} 
 where $\norm{V}_{L^1_2}:=\int (1+|x|)^2 |V|(dx)$ and $C$ is  some absolute positive constant. 
\item[(b)] For any $\phi\in \sH_{ac}$, there exist $\phi_\pm \in \sH=L^2$ such that 
\begin{align}\label{scattering:U0(t)} 
 &U(t)\phi =  U_0(t) {\phi_\pm}  +o(1)\quad in \  L^2  \quad as \ t\to \pm\iy \,.
\end{align}
Moreover,
\begin{align}\label{U0:scatter:phi+} 
 &U(t)\phi (x)= \frac{1}{\sqrt{2it}} e^{i \frac{x^2}{4t}} \widehat{\phi_\pm}(\pm\frac{x}{2|t|})  +o(1)\quad in \  L^2  \quad as \ t\to \pm\iy \,,
\end{align} 
\end{enumerate}  
where $\hat{f}=\Fo_0(f)$ is the usual Fourier transform (\ref{F0:fourier}).
\end{theorem}  

\begin{proof}[Proof of Theorem \ref{t:wave:scatter}]  
(a) Let $\phi \in \sH_{ac}$. 
From (\ref{spec:F*Fo}) and Theorem \ref{Fo:waV} 
we have 
\begin{align*}
&  U(t) \phi = \Fo_\pm^* e^{ -i t k^2} \Fo_\pm \phi\\
& W_\mp U_0(t)\phi= \Fo_\pm^* \Fo_0 \Fo_0^* e^{-i t k^2} \Fo_0 \phi= \Fo_\pm^*  e^{-i t k^2} \Fo_0 \phi \\
& U(t) \phi- W_\mp U_0(t)\phi =\Fo_\pm^*  e^{-i t k^2} \left( \Fo_\pm -\Fo_0 \right)\phi\,.
\end{align*}  
Since $\Fo_\pm$ is surjective $\sH_{sc}\to L^2$, there exists $\eta_\mp$ in $\sH_{sc}$ such that $\Fo_\pm (\eta_\mp)=( \Fo_\pm -\Fo_0)\phi$. 
Hence, 
\begin{align*}
& U(t) \phi- W_\mp U_0(t)\phi =\Fo_\pm^*  e^{-i t k^2}\Fo_\pm (\eta_\mp)=U(t) \eta_\mp \,.
\end{align*}  
Here we can estimate $\Fo_+ (\eta_-)$ from (\ref{Fo+:f}) and (\ref{psiV:L-S})
  \begin{align*}
&\left\vert\Fo_+ (\eta_-)(k)\right\vert= \left\vert\frac{i}{2\sqrt{2\pi} |k|} \int  \overline{\psi^+(y,k)}  V(dy) \int \phi(x) e^{- i |k| | x-y|}dx\right\vert \\
\le& \frac{C}{1+|k|}\norm{V}_{L^1_2} \norm{\phi}_{L^1_1}\,.
\end{align*} 
Similarly, we have 
 \begin{align*}
&\left\vert\Fo_- (\eta_+)(k)\right\vert 
\le \frac{C}{1+|k|}\norm{V}_{L^1_2} \norm{\phi}_{L^1_1}\,.
\end{align*} 
On the other hand, it is easy to check the proof of \cite[Theorem 5.1]{GH98}  to derive that $|{\psi^\pm(y,k)}|\le C(1+|y|) e^{ c_0\norm{V}_{L^1_1}}$, see also \cite{Zheng2010i} for the case $V$ being a  function in $L^1_2$. 
Therefore, we have established part (a).  

\bigskip 
(b)  
Given $\phi\in \sH_{ac}$, we first show: There exists $\phi_+$ in $L^2$ such that (\ref{scattering:U0(t)}) holds. 
In fact,  since $W_+$ has range $\sH_{ac}$, 
we can take ${\phi}_+\in L^2$ such that  
\begin{align}\label{W:phi+}
W_{+} ({\phi}_+) =\phi\iff W_{+}^* ({\phi}) =\phi_+\,. 
\end{align}
Then,  from the existence of $W_+$ in Theorem \ref{eWav:pm} we have  as  $ t\to +\iy$
\begin{align*}
& e^{it H} e^{-it H_0} \phi_+  -W_+ \phi_+=o(1) \quad in \ L^2\,,   \end{align*} 
{that is}, 
$ e^{-it H_0}  \phi_+ -e^{-it H} W_+ \phi_+=o(1)$. 
Now, we deduce from the identity (\ref{W:phi+}) 
\begin{align}
 &e^{-it H} \phi=e^{-it H_0} W^*_+\phi+o(1)\notag\\
& e^{-it H}\phi =e^{-it H_0} {\phi_+} +o(1)  \quad \text{ in $L^2$ as  $t\to +\iy$.}  
\end{align} 
Similarly, we can show that there exists $\phi_-$ in $L^2$
\begin{align*}
 e^{-it H}\phi =e^{-it H_0} {\phi_-} +o(1)  \quad \text{in $L^2$ as  $t\to -\iy$}   
\end{align*}
where $W_-\phi_-=\phi\iff  \phi_-=W_-^* \phi$.  
This proves (\ref{scattering:U0(t)}) via the intertwining property $U(t)W_\pm= W_\pm U_0(t)$. 

To conclude the proof of the theorem, we further prove (\ref{U0:scatter:phi+}). 
 But this follows from  the asymptotics for the free waves  (Fraunhofer formula) for all $f\in L^2$
\begin{align*}
\lim_{t\to \pm\iy} e^{-it H_0}f(x) = \frac1{\sqrt{2it } } e^{ i \frac{x^2}{4t}} \hat{f}(\pm\frac{x}{2|t|})   
   \qquad \text{in $L^2$}\,,
\end{align*} 
where the function $\sqrt{z}$ is assumed on the branch so that $\Re\sqrt{z}\ge 0$, 
or equivalently, $\sqrt{z}$ is defined analytically in $\arg z\in (-\frac{\pi}{2}, \frac{\pi}{2})$,
see e.g. \cite[Theorem IX.31]{RS}. 
\end{proof} 

\begin{remark} An alternative formula of (\ref{U(t):V-W}) is true: 
\begin{align}\label{eU(t):W} 
& U(t) \phi = W_\mp U_0(t) {\phi_\mp}  \qquad \text{in $ L^2$ for all $t$}.
\end{align} 
Indeed, applying the intertwining equation (\ref{eW:intertwine}) for $U(t)P_{ac}=W_\mp U_0(t)W_\mp^*\,$, 
we calculate: For all $\phi\in \sH_{ac}\,$,
\begin{align*}
&U(t)\phi= W_\mp U_0(t) \phi_\mp \qquad \text{ in  $L^2$},
\end{align*} 
where $W_\mp^*\phi= {\phi_\mp}\in L^2$.  
\end{remark} 

\begin{remark} 
If $V$ is a real function in $L^1_2$, then  the  dispersive 
estimates hold for all $\eta\in \sH_{ac}$  \cite{GSch04} 
\begin{equation}\label{eU(t):time-decay}
 \norm{U(t)\eta}_{L^p(\R)} \lesssim 
 \frac{C}{|t|^{\frac12-\frac1p } }\norm{\eta}_{L^{p'}(\R)} \ \qquad p'\in [1,2]   
\end{equation} 
where 
$p$ and $p'$ are H\"older conjugate exponents.  
This suggests that, in view of part (a) in Theorem \ref{t:wave:scatter}, 
  there holds the approximation  as $t\to \pm\iy$,
 \begin{align}\label{eU(t):WU0} 
&U(t)\phi=e^{-it H} \phi \approx W_\mp U_0(t)\phi  \qquad \text{in $L^p$ for $p>2$}.  
\end{align} 
Regarding recent progress on  dispersive estimates we also refer to  \cite{EgoKopyMaTe16} 
for $V\in L_1^1$ on the line as well as   \cite{ErGoGr21,Gold12m} for measure-valued potentials in higher dimensions.   
The dispersive estimate (\ref{eU(t):time-decay}) should be valid as well with $V$ being a finite measure (\ref{eV:wei2}). 
\end{remark}

We have seen that the proofs of the scattering for $H=H_V$ in Theorem \ref{t:wave:scatter} as well as the $L^2$ theory with eigenfunction expansions 
are built on the construction of spectral measure and the asymptotic completeness of wave operators. 
The scattering theorem continues to hold true as soon as (\ref{eWav:pm}) is valid. 







\begin{lemma}\label{l:exist:ac}  Let $V\in L^1\cap L^2$ be real-valued. Then 
$(H-z)^{-1}-(H_0-z)^{-1} $ is trace class for some nonreal $z$. 
Hence $\Om_\pm$ exist in \eqref{e:Omega-f} and are asymptotic complete. 
The continuous part of $H$ and $H_0$ are unitarily equivalent. 
\end{lemma}  

\begin{proof} This lemma follows from \cite[Theorem 4.12, chapter X]{Kato66}.   
Alternatively, one can also derive the existence and completeness of $\Om_\pm$ from  
a pattern of  Kuroda-Birman's theorem   
\cite[Theorem XI.9]{RS}.   
\end{proof} 


\begin{theorem}\label{th:Omega} Suppose $V\in L^1\cap L^2(\R)$ is real-valued. The following scattering properties hold.  
\begin{enumerate}
\item[(a)] If $\phi\in \sH_{ac}$, then there exists $\eta_\mp\in \sH_{ac}$ such that 
\begin{align*}
&U(t)\phi = \Om_\mp U_0(t)\phi  +U(t)\eta_\mp\qquad  \text{in $L^2$} 
\end{align*} 
\item[(b)] For any $\phi\in \sH_{ac}$, there exist $\phi_\pm \in \sH=L^2$ such that 
\begin{align*}
 &U(t)\phi =  U_0(t) {\phi_\pm}  +o(1)\quad in \  L^2  \quad as \ t\to \pm\iy \,.
\end{align*}
Moreover,
\begin{align*}
 &U(t)\phi (x)= \frac{1}{\sqrt{2it}} e^{i \frac{x^2}{4t}} \widehat{\phi_\pm}(\pm\frac{x}{2|t|})  +o(1)\quad in \  L^2  \quad as \ t\to \pm\iy \,
\end{align*} 
where $  \phi_\mp=\Om_\mp^* \phi\,$. 
\end{enumerate} 
\end{theorem} 

 According to Lemma \ref{l:exist:ac} and Proposition \ref{c:wave:F*F0}, the wave operators $\Om_\mp$ verify (\ref{eWav:pm}):
\begin{align*}
&\Omega_\mp f: = \lim_{t \rightarrow \mp\infty} e^{itH} e^{-itH_0} f\, = \cF^*_\pm \cF_0\,, 
\end{align*} 
so that $\ran \Om_\mp=\sH_{ac}$. It follows that  
and the perturbed plane wave is formally given by   $e_\pm(x,\xi)=\Om_\mp e^{\pm i kx} $. Moreover, there are the intertwining properties: 
  \begin{enumerate}
\item[(i)] $R(\zeta)\Om_\mp= \Om_\mp R_0(\zeta)$ \quad for $\zeta\in \C\setminus [-\lam_0, +\iy)$, 
where $-\lam_0=\inf \sigma(H)$, the infimum of the spectrum of $H$.  
\item[(ii)] $\vphi(H)\Om_\mp= \Om_\mp \vphi(H_0)$.
\end{enumerate} 
The scattering matrix $S=\Om^*_+\Om_-$  is unitary.  
Examples on the spectral and scattering properties  of wave operators and $S$-matrix that include sign-changing barrier potential,  Coulomb potential, nonlocal $\de$ and $\de'$ interactions  can be found in e.g. \cite{BZ2010, EcKMTe2014,EgHoTe2015,GlowKu2020,Teschl2014b}. 
\begin{remark} For $V$ in $L^1\cap L^2$, the proofs in Lemmas \ref{l:e(xxi):E0}, \ref{lem:F-Green} and \ref{lem:Phi} do not show the (uniform) 
$L^\iy$ boundedness or integrability in the $\xi$-variable for $e(x,\xi)$. However, in light of Remark \ref{e(xxi):alpha-beta},
 $e(x,\xi)$ is uniformly bounded  for $x \in \R$,  $\alpha \leq \vert \xi \vert \leq \beta$ 
with $ 0 < \alpha < \beta$ and $[\alpha,\beta] \cap {\cal E}_0^2 =\emptyset$.
This observation allows us to verify (\ref{Psi(k):cE0}) for all functions $f\in \cD(H_0)$ with support away from $\cE$. 
As a consequence, we can strengthen  the property (i) for all  $\zeta\in\C\setminus \cE^2$ by a density argument. 
\end{remark}

\vspace{.2in}
 
\noindent 
{\em Acknowledgments} 
The author  thanks \mbox{Michael Goldberg} for valuable comments regarding the dispersive estimates for measure-valued potentials in weighted $L^1$
class. {The author also thanks the referees for careful reading of the manuscript and providing helpful comments}.  
The research  of the  author is supported in part by  {NSFC} 12071323 as Co-PI.

\end{document}